\documentclass{asl}

\usepackage{enumerate}
\usepackage{braket}
\usepackage[colorlinks]{hyperref}
\usepackage{caption}
\usepackage{microtype}

\newenvironment{claimproof}[1][Proof of Claim]{\begin{proof}[#1]\renewcommand*{\qedsymbol}{\(\lozenge\)}}{\end{proof}}

\newtheorem{lemma}{Lemma}[section]
\newtheorem*{lemma*}{Lemma}
\newtheorem{theorem}[lemma]{Theorem}
\newtheorem*{theorem*}{Theorem}
\newtheorem{corollary}[lemma]{Corollary}
\newtheorem{proposition}[lemma]{Proposition}

\newtheorem*{proposition*}{Proposition}

\newtheorem*{fact*}{Fact}

\newtheorem*{notation*}{Notation}
\newtheorem*{conventions*}{Conventions}
\newtheorem{remark}[lemma]{Remark}
\newtheorem*{remark*}{Remark}

\newtheorem*{corollary*}{Corollary}
\newtheorem{conjecture}{Conjecture}
\newtheorem*{conjecture*}{Conjecture}
\newtheorem{problem}{Problem}
\newtheorem*{problem*}{Problem}
\newtheorem{question}{Question}
\newtheorem*{question*}{Question}

\newtheorem{assumption*}{Assumption}

\theoremstyle{definition}
\newtheorem{example}{Example}
\newtheorem*{example*}{Example}
\newtheorem{definition}[lemma]{Definition}
\newtheorem*{definition*}{Definition}

\theoremstyle{remark}

\newtheorem*{claim*}{Claim}

\newtheorem*{case*}{Case}

\newcommand{\Z}{\mathbb{Z}}

\newcommand{\Q}{\mathbb{Q}}

\newcommand{\bs}{\backslash}

\newcommand{\satisfies}{\vDash}

\renewcommand\AA{{\mathcal A}}
\newcommand\BB{{\mathcal B}}
\newcommand\CC{{\mathcal C}}

\newcommand\LL{{\mathcal L}}
\newcommand\MM{{\mathcal M}}

\renewcommand\SS{{\mathcal S}}

\newcommand\acl{\hbox{\rm acl}}
\newcommand\into{\hookrightarrow}
\newcommand\contains{\supseteq}

\newcommand\<{\langle}
\renewcommand\>{\rangle}
\newcommand\az{{\aleph_0}}
\newcommand\cont{2^{\aleph_0}}
\newcommand{\Age}{\mathfrak A}

\newcommand{\tp}{\mathrm{tp}}

\def\abar{\bar{a}}

\def\bbar{\bar{b}}
\def\cbar{\bar{c}}
\def\dbar{\bar{d}}
\def\ebar{\bar{e}}

\def\gbar{\bar{g}}
\def\hbar{\bar{h}}

\def\mbar{\bar{m}}
\def\nbar{\bar{n}}

\def\wbar{\bar{w}}
\def\xbar{\bar{x}}
\def\ybar{\bar{y}}
\def\zbar{\bar{z}}
\def\A{{\mathcal A}}
\def\B{{\mathcal B}}

\def \F{{\mathcal F}}
\def\Aind{{N^{{\rm ind}}}}

\newcommand{\Arr}{\mathrm{Supp}}
\newcommand{\QF}{\mathrm{QF}}

\title{Counting siblings in universal theories}
\author{Samuel Braunfeld}
\revauthor{Braunfeld, Samuel}
\address{Department of Mathematics\\
	University of Maryland, College Park\\
	College Park, MD 20742, USA}
\email{sbraunf@umd.edu}
\author[Michael C. Laskowski]{Michael C. Laskowski$^*$}
\revauthor{Laskowski, Michael C. }
\address{Department of Mathematics\\
	University of Maryland, College Park\\
	College Park, MD 20742, USA}
\email{laskow@umd.edu}
\thanks{$^*$Partially supported
	by NSF grant DMS-1855789}

\begin{document}

\begin{abstract}
We show that if a countable structure $M$ in a finite relational language is not cellular, then there is an age-preserving $N \supseteq M$ such that $\cont$ many structures are bi-embeddable with $N$. The proof proceeds by a case division based on mutual algebraicity.
\end{abstract}

\maketitle

\section{Introduction}

The model-theoretic condition of cellularity has appeared several times as a dividing line in the complexity of universal theories, including when counting the number of countable models \cite{MPW}, counting the number of finite models as a function of size \cite{LT2}, and counting the number of non-isomorphic substructures of countable models \cite{LM}. In this paper, we present a general approach to proving results about cellularity via another model-theoretic condition, mutual algebraicity. The approach is to first prove that the non-mutually algebraic case is wild, likely using the Ryll-Nardzewski-type characterization of mutual algebraicity from \cite{LT1}. In a companion paper \cite{cma}, we characterize the mutually algebraic non-cellular case. As mutually algebraic structures admit a nice structural decomposition, it is relatively quick to prove the mutually algebraic non-cellular case is still wild. This approach was already largely present in \cite{LT2}, and we apply it here to the question of counting siblings. 

We call two (not necessarily elementarily) bi-embeddable structures {\em siblings} ($f \colon M \into N$ is an embedding if $R(x_1, \dots, x_n) \iff R(f(x_1), \dots, f(x_n))$ for every atomic relation $R$). Given a countable relational structure $M$, our goal is to count the number of siblings of $M$, up to isomorphism. Thomass\'{e} has conjectured the following, counting $M$ as a sibling of itself.

\begin{conjecture} [Thomass\'{e}, \cite{Thom}]
Given a countable structure $M$ in a countable relational language, $M$ has either 1, $\az$, or $\cont$ siblings, up to isomorphism.
\end{conjecture}

This conjecture has been proven in the case of linear orders \cite{chains}, the gap from 1 to $\az$ proven for $\az$-categorical structures by making use of the monomorphic decomposition \cite{sib}, and the gap from 1 to $\az$ proven for cographs \cite{cograph}. The gap from 1 to $\az$ has also been conjectured in the case of graphs, connected graphs where the siblings must also be connected \cite{twins}, and trees where the siblings must also be trees (as opposed to forests) \cite{BT}, and some partial results obtained in these cases.

If two structures are siblings, they must have the same finite substructures, and so satisfy the same universal theory. Thus, we may coarsen Thomass\'{e}'s conjecture to considering the maximum number of siblings of any model of a given universal theory, which may be viewed as a measure of complexity of that theory. Indeed, for a model to have many siblings, we must produce non-isomorphic structures that look somewhat alike (the similarity required of siblings may be increased by requiring {\em elementary} bi-embeddability, as in \cite{SB}). Complex theories will allow their models to be nuanced enough to admit many siblings. Uncomplicated theories will not allow for such nuance, and so whenever models look alike, they will be the same. (For example, the theory of $n$ disjoint unary predicates, where models are isomorphic once the cardinalities of the predicates match.) The complexity gaps of Thomass\'{e}'s conjecture then call to mind model-theoretic dividing lines. 

However, we note that it is possible for individual structures to be very complicated, yet have few siblings. For example, $\omega$ with successor has only itself as a sibling. Thus the same is true of any expansion, in particular the expansion by the graphs of addition and multiplication. So it is difficult to see how model theory will inform the full conjecture.

Our main theorem confirms the weakening of Thomass\'{e}'s conjecture to the level of universal theories in a finite relational language.
\begin{theorem}[Theorem \ref{thm:main2}] \label{thm:main}
Let $T$ be a universal theory in a finite relational language. Then one of the following holds.
\begin{enumerate}
\item $T$ is finitely partitioned. Every model of $T$ has one sibling.
\item $T$ is cellular. The finitely partitioned models of $T$ have one sibling and the non-finitely partitioned models have $\aleph_0$ siblings.
\item $T$ is not cellular. For every non-cellular $M \models T$, there is some $N \supseteq M$ such that $N \models T$ and $N$ has $2^{\aleph_0}$ siblings. Furthermore, if $T$ is mutually algebraic, we may take $N \succeq M$.
\end{enumerate}
\end{theorem}

Theorem \ref{thm:main} does have implications at the level of structures, confirming some conjectures of \cite{sib}.

\begin{corollary}[Corollary \ref{cor:universal2}] \label{cor:universal}
Let $M$ be a countable model in a finite relational language that is universal for its age. Then one of the following holds.
\begin{enumerate}
\item $M$ is finitely partitioned, and has one sibling.
\item $M$ is cellular but not finitely partitioned, and has $\aleph_0$ siblings.
\item $M$ is not cellular, and has $\cont$ siblings.
\end{enumerate}
\end{corollary}

This also implies the result for $\omega$-categorical $M$ in a finite relational language, since then we may pass to its model companion. Example \ref{ex:inf lang} shows Corollary \ref{cor:universal} does not hold for infinite relational languages with finite profile.

We close with some comments connecting our results to previous work on cellularity. First, we note that Theorem \ref{thm:main} is a refinement of the main result of \cite{MPW} that non-cellular universal theories have $\cont$ non-isomorphic models. Second, Corollary \ref{cor:universal} may be seen as a dual to the main result of \cite{LM} that an atomically stable non-cellular countable structure has $\cont$ non-isomorphic substructures. When $M$ is universal for its age, as in Corollary \ref{cor:universal}, siblings are equivalent to age-preserving extensions, and we again see cellularity is the dividing line between $\az$ and $\cont$.

\subsection{Proof sketch} \label{sub:sketch}
The primary intuition behind the proof of the main theorem is that if a universal theory $T$ is non-cellular, then either it is unstable and so has a model encoding $(\Q, <)$, or has a model that in some sense encodes a partition with infinitely many infinite parts. We present three examples corresponding to the three cases of our proof, and explain how to obtain $\cont$ many siblings in each.

\begin{enumerate}
\item Let $M = (\Q, \leq)$. Then any countable non-scattered order is a sibling of $M$, and there are $\cont$ many.
\item Let $M$ be an equivalence relation with infinitely many infinite classes. Then we may pass to an elementary extension $M^* \succ M$ containing infinitely many new infinite classes $\set{\AA_q:q \in \Q}$. For each injective $f\colon \Q \to \omega$, let $M_f$ be obtained by cutting down each $\AA_q$ to size $f(q)$. Then each $M_f$ is a sibling of $M$, and they are pairwise non-isomorphic, as they have distinct sizes of finite classes.
\item Let $M = (\omega, s)$, where $s$ is the successor relation. We first pass to an elementary extension $M' \succ M$ containing infinitely many $\Z$-chains. Then, as in case (2), we may pass to a further elementary extension $M^* \succ M'$ containing infinitely many new $\Z$-chains $\set{\AA_q:q \in \Q}$. For each injective $f\colon \Q \to \omega$, we let $M_f$ be obtained by cutting down each $\AA_q$ to a connected piece of size $f(q)$.
\end{enumerate}

Our proof follows these three examples. The bulk of the work is in generalizing Case 2 to the setting of a non-mutually algebraic $M$. The role played by equivalence classes is generalized to that of $k$-cliques in \S \ref{sec:kcliq}, while \S \ref{sec:oi} guarantees that if we cannot add such $k$-cliques to $M$, then we may find many siblings as in Case 1. Otherwise, for $M$ non-mutually algebraic, we may generalize the proof of Case 2 by adding infinitely many $k$-cliques to $M$, which is done in \S \ref{sec:full}-\ref{sec:notma}. Finally, for $M$ mutually algebraic but non-cellular, we generalize Case 3 in \S \ref{sec:ma}.   

\subsection*{Acknowledgments} The authors are grateful to the anonymous referee for their careful reading and pointing out inaccuracies in a preliminary version of this paper.

\section{Conventions and background}

The following conventions will be in effect throughout this paper, unless otherwise noted.

\medskip

\noindent{\bf$M$ is a countable structure in a finite relational language $\LL$.}

\medskip

\noindent{\bf Types are quantifier-free types, and indiscernibility is with respect to quantifier-free formulas.}

\medskip

We now briefly cover the definitions and results from elsewhere that we will need.

\begin{definition}
A structure $M$ is {\em finitely partitioned} if it admits a finite partition $\set{C_1, \dots, C_n}$ such that $\Pi_i Sym(C_i) \subset Aut(M)$.
\end{definition}

\begin{definition} \label{def:cell}
A structure $M$ is {\em cellular} if, for some $n$ and $k_1, \dots, k_n \in \omega$, it admits a partition $K \sqcup \set{\cbar_{i,j} | i \in [n], j \in \omega}$ satisfying the following.
\begin{enumerate}
\item $K$ is finite, and each $\cbar_{i,j}=(c_{i,j}^1,\dots, c_{i,j}^{k_i})$ has length $k_i$.
\item For every $i \in [n]$ and $\sigma \in S_\infty$, there is a $\sigma_i^* \in Aut(M)$ mapping each $\cbar_{i,j}$ onto $\cbar_{i, \sigma(j)}$ by sending $c_{i,j}^\ell$ to $c_{i, \sigma(j)}^\ell$ for $1 \leq \ell \leq {k_i}$, and fixing $M \bs \bigcup_{j \in \omega} \cbar_{i, j}$ pointwise.
\end{enumerate}

We call such a partition a {\em cellular partition}.
\end{definition}

\begin{example} \label{ex:cell}
Let $M$ be a graph consisting of infinitely many disjoint edges and an infinite clique. Then $M$ is cellular -- we may take $K = \emptyset$, $n=2$, let $\set{\cbar_{0,j}:j\in\omega}$ enumerate the disjoint edges, and $\set{c_{1,j}:j\in\omega}$ enumerate the clique.
\end{example}

Note $M$ is finitely partitioned if and only if $M$ is cellular as witnessed by a partition with each $k_i=1$.  The following definitions are from \cite{MA}, which builds on results from
\cite{MA0}.

\begin{definition} \label{def:ma}
Given a structure $M$ and $n\ge 1$, a set $S\subseteq M^n$ is  {\em mutually algebraic} if there is some $K\in\omega$ such that  $|\{\abar\in S:m\in\abar\}|\le K$ for every $m\in M$.
Let $\LL_M$ be $\LL$ expanded by constant symbols for every element of $M$, and $M_M$ the natural expansion of $M$ to $\LL_M$. An $\LL_M$-formula $\phi(x_1,\dots,x_n)$ is mutually algebraic if it defines a mutually algebraic subset of $ M^n$. We then let $\MM\AA^*(M)$ be the smallest set of $\LL_M$-formulas containing the mutually algebraics,  closed  under adjunction of dummy variables and boolean combinations.

Finally, we say $M$ is {\em mutually algebraic} if every $\LL_M$-formula is equivalent to a formula in  $\MM\AA^*(M)$.
\end{definition}

Note that every unary relation is mutually algebraic. Less obviously, cellular structures are mutually algebraic.

\begin{lemma}
  Let $M$ be mutually algebraic and $N \subset M$ a substructure. Then $N$ is mutually algebraic.
\end{lemma}
	\begin{proof}  Let $(M,N)$ be the expansion of $M$ formed by adding a unary predicate $U$ interpreted as $N$.
		Let $\Aind$ denote the expansion of $N$ by relations $P_D$ naming the trace $D\cap N^n$ of every $(M,N)$-definable (with parameters) subset $D\subseteq M^n$, for all $n$. As the set $N$ is definable in $(M,N)$, it is easily checked that $\Aind$ admits elimination of quantifiers.  Moreover, every parameter-definable set
		of $\Aind$ is 0-definable in $\Aind$, and is definable in $(M,N)$.  
		
		\begin{claim*}  $\Aind$ is mutually algebraic.
		\end{claim*}
		
		\begin{claimproof}  We show that every $\Aind$-definable subset $B\subseteq N^n$ is in $\MM\AA^*(\Aind)$.  Since mutual algebraicity is preserved under unary expansions by 
			Theorem 3.3 of \cite{MA}, $(M,N)$ is mutually algebraic, and so $B$ is in $\MM\AA^*((M, N))$, as witnessed by a boolean combination of  sets $\set{Y_1,\dots, Y_m}$,
			each realizing an adjunction of a mutually algebraic formula by dummy variables.
			 As the same is true for each $Y_i\cap N^n$, $B\in\MM\AA^*(\Aind)$.
		\end{claimproof}
		
		It is easily checked that the $L$-structure $N$ is a reduct of $\Aind$, hence $N$ is mutually algebraic by Corollary 7.4 of \cite{LT1}.
	\end{proof}

In addition to mutual algebraicity, the properties of being finitely partitioned and cellular are preserved under passing to a substructure. Thus, they are properties of a universal theory, and so we will say a universal theory $T$ has one of these properties if all of its countable models do.

 We record one additional characterization of mutual algebraicity.  

\begin{theorem} \cite{LT1}*{Theorem 2.1} \label{thm:ma at}
$M$ is mutually algebraic if and only if every atomic $\LL$-formula is $Th(M_M)$-equivalent to a boolean combination of quantifier-free mutually algebraic $\LL_M$-formulas.
\end{theorem}

\begin{example}
Consider a structure $(M, E)$ where $E$ is an equivalence relation with $n$ classes, each class infinite. Then the relation $E$ is not mutually algebraic. However, using the constants $m_1, \dots, m_n$ to name one element from each class, we have $E(x, y) \iff \bigvee_i(E(x,m_i) \wedge E(y,m_i))$, which is a boolean combination of quantifier-free mutually algebraic $\LL_M$-formulas. Thus $M$ is mutually algebraic.
\end{example}

\begin{definition}
Given a set $A$, let $\QF_k(A)$ be the set of quantifier-free formulas over $A$ with $k$ variables.

Given a structure $M$, $\cbar \in M^k$, and $A \subset M$, the {\em type of $\cbar$ over $A$} is $\tp(\cbar/A) = \set{\theta(\xbar) \in \QF_k(A) : M \models \theta(\cbar)}$.

Given a structure $M$, a {\em $k$-type over $M$} is some $p(\xbar) \subset \QF_k(M)$ such that there is some elementary extension $N \succ M$ and $\nbar \in N^k$ such that $p(\xbar) = \tp(\nbar/M)$.
\end{definition}

\begin{definition}
Given a structure $M$ and a $k$-type $p$ over $M$, we say {\em $p$ supports an infinite array} if there is some  $N \succ M$ and a set of pairwise disjoint $k$-tuples $\set{\nbar_i \in N^k : i \in \omega}$ such that $\nbar_i \models p$, for every $i$.

We let $\Arr_k(M)$ denote the set of $k$-types over $M$ that support infinite arrays.

We say $p(\xbar)$ is {\em coordinate-wise non-algebraic} if $(x_i\neq b)\in p$ for every $x_i\in \xbar$ and every $b\in M$.
\end{definition}

\begin{lemma} \label{char} Let $M$ be any structure, and $p(\xbar)$ a type over $M$. Then $p \in \Arr_k(M)$
if and only if $p(\xbar)$ is coordinate-wise non-algebraic.
\end{lemma}
\begin{proof}
If $(x_i=b)\in p$ for some $x_i$ and some $b\in M$, then any two realizations of $p$ have non-empty intersection, so $p$ does not support an infinite array
(or an array of length 2, for that matter).  Conversely, assume $p$ is coordinate-wise non-algebraic, but $p$ does not support an infinite array.
By compactness, there is some $n$ and some $\theta(\xbar)\in p$ such that
in $M$, there do not exist $n$ pairwise disjoint realizations of $\theta$.  Among all such, choose $\theta$ so that $n$ is minimized, and choose $\set{\bbar_i:i<n}$ from $M$,
pairwise disjoint with $M\models\theta(\bbar_i)$ for each $i$.  Choose $M^*\succeq M$ and $\abar$ from $M^*$ realizing $p$.
As $p$ is coordinate-wise non-algebraic, $\abar$ is disjoint from $M$, hence disjoint from each $\bbar_i$.  Thus $\set{\abar}\cup\set{\bbar_i:i<n}$ gives $(n+1)$ pairwise disjoint realizations of
$\theta(\xbar)$, which is impossible since $M^*\succeq M$.
\end{proof}

\begin{theorem} \label{thm:ma infinite arrays}
If $M$ is not mutually algebraic, then there is some $M' \succ M$ and some $k \in \omega$ such that $\Arr_k(M')$ is infinite.
\end{theorem}
\begin{proof}
By \cite{LT1}*{Theorem 6.1}, there is some countable $M^* \equiv M$ and some $k$ such that $\Arr_k(M^*)$ is infinite. Let $M'$ elementarily embed $M$ and $M^*$. By compactness, every $p \in \Arr_k(M^*)$ extends to some $p' \in \Arr_k(M')$.
\end{proof}

\begin{definition}
Fix a structure $M$, let $S = (\bbar_i \in M^k : i \in (I, <))$ be a sequence of $k$-tuples, and let $A \subset M$. $S$ is {\em order indiscernible over $A$} if $\tp(\bbar_{i_1}, \dots, \bbar_{i_n}/A) = \tp(\bbar_{j_1}, \dots, \bbar_{j_n}/A)$ whenever $i_1 < \dots < i_n$ and $j_1 < \dots < j_n$ (where, by our convention, $\tp$ is understood to mean quantifier-free type).

$S$ is {\em totally indiscernible over $A$} if $\tp(\bbar_{i_1}, \dots, \bbar_{i_n}/A) = \tp(\bbar_{j_1}, \dots, \bbar_{j_n}/A)$ whenever $i_1, \dots, i_n$ are pairwise distinct, as are $j_1, \dots j_n$.

$S$ is {\em strictly order indiscernible over $A$} if it is order indiscernible over $A$ but not totally indiscernible over $A$.
\end{definition}

\begin{definition}
A countable structure $M$ is {\em universal for its age} if every other countable structure with the same age embeds into $M$. Equivalently, $M$ is countable universal for its universal theory.
\end{definition}

\section{Strictly order indiscernible arrays} \label{sec:oi}

As we are aiming to prove that cellularity is the dividing line between having a model with $\az$ and $\cont$ siblings, we expect non-stability, as manifested by an infinite strictly order-indiscernible sequence of $k$-tuples, to provide a model with $\cont$ siblings. We prove this in the case of infinite arrays, but first we need a definition and easy lemma.

\begin{definition} \label{arraymin}
For $M$ non-mutually algebraic, $M$ is {\em array-minimal of index $k$} if $\Arr_k(M)$ is infinite and there does not exist a $k'<k$ and an age-preserving $N\supseteq M$ for which $\Arr_{k'}(N)$ is infinite.
\end{definition}

\begin{example}
Consider the structure $M = (\Q\times \set{0,1}, \prec, E)$, where $E$ is a binary relation such that $(q,i)E(r,j)$ iff $q=r$ and $i \neq j$, and $\prec$ is a quaternary relation encoding the usual $\leq$ relation between pairs of $E$-connected points. Then there is only 1 coordinate-wise non-algebraic 1-type over $M$, namely the type of an isolated point. The same will be true for any age-preserving $N \contains M$. However, there are infinitely many coordinate-wise non-algebraic 2-types over $M$ -- into any cut of $M$, we may insert an $E$-related pair of points. Thus $M$ is array-minimal of index 2.
\end{example}

\begin{lemma} \label{ez}   If $M$ is not mutually algebraic, then for some $k\ge 1$, there is an age-preserving $M'\supseteq M$ that is array-minimal of index $k$.
Moreover, for every elementary extension $M^*\succeq M'$ and for any substructure $N$ with $M'\subseteq N\subseteq M^*$, $N$ is also array-minimal of index $k$.
\end{lemma}
\begin{proof}

As $M$ is not mutually algebraic, by Theorem \ref{thm:ma infinite arrays} there is some age-preserving $N \contains M$ and some $\ell \in \omega$ such that $\Arr_\ell(N)$ is infinite. Among all age-preserving extensions of $M$, there is one with the least $k$ such the extension has infinitely many $k$-types that support infinite arrays, and choose that extension to be $M'$.

For the moreover clause, choose any $M'\subseteq N\subseteq M^*$
with $M^*\succeq M'$.  Every $p\in \Arr_k(M')$ has an extension $p^*\in \Arr_k(M^*)$.
As the restriction of each of these types $p^*$ to a type over $N$ also supports an infinite array, $N$ is also array-minimal of index $k$.
\end{proof}

\begin{proposition} \label{one} Suppose $M$ is not mutually algebraic, $M$ is array-minimal of index $k$, and that some 
$p\in \Arr_k(M)$ supports an infinite array $\set{\abar_i:i\in\omega}$ that is strictly order indiscernible over $M$.  Then there is an age preserving $N\supseteq M$ with $2^{\aleph_0}$ siblings.
\end{proposition}

\begin{proof} From our assumption on $p$ and compactness, choose an elementary extension $M^*\succeq M$ containing a strictly order-indiscernible array $A=\set{\abar_j:j\in\Q}$ of realizations of $p$. Let $N$ be the substructure of $M^*$ with universe $M\cup A$, and let
  $N^*=M\cup\set{\abar_j:j\le 0}\cup\set{\abar_j:j\ge 1}$.
Choose a family $\F=\set{J_\alpha:\alpha\in 2^\omega}$ of subsets of $(0,1)\cap\Q$ such that the ordered structures $(J_\alpha,\le)$ are pairwise non-isomorphic and each embed $(\Q, \leq)$.
For each $\alpha$, let $N_\alpha\subseteq N$ have universe $N^*\cup\set{\abar_j:j\in J_\alpha}$.
As $(J_\alpha, \leq)$ and $(J_\beta, \leq)$ both embed $(\Q, \leq)$, they are bi-embeddable, and these lift to bi-embeddings of $N_\alpha$ and $N_\beta$ fixing $N^*$ pointwise.

It is true that some of the structures $N_\alpha,N_\beta$ may be isomorphic, but we will find a subfamily of size $2^{\aleph_0}$ that are pairwise non-isomorphic, which 
finishes our argument.  Our method will be to prove that for any given $N_\alpha$, $\set{N_\beta : N_\beta\cong N_\alpha}$ is countable, which suffices.  
In particular, we will fix a uniform finite set $F \subset N^*$ and prove that when $\alpha\neq\beta$, there is
no isomorphism $h:N_\beta\rightarrow N_\alpha$ that fixes $F$ pointwise. Then we cannot have $h \colon N_\beta\rightarrow N_\alpha$ and $h' \colon N_{\beta'}\rightarrow N_\alpha$ with $h(F) = h'(F)$ pointwise, since $h^{-1} \circ h'$ would fix $F$.  As each $N_\alpha$ is countable, there are only countably many possible
images of $F$ under an isomorphism $h \colon N_\beta\rightarrow N_\alpha$, hence $\set{\beta:N_\beta\cong N_\alpha}$
is countable, as required.

Constructing $F$ and proving its suitability will take the rest of the section.
\end{proof}

To begin, we have the following definition that involves permutations of $k$-tuples.
For a given $k$-tuple $\abar_q$ from $N$ and a given $\pi\in Sym(k)$, let $\pi(\abar_q)$ be the permutation of $\abar$ induced by $\pi$.

\begin{definition} \label{def:permis}
 Working in $N$, a permutation $\pi\in Sym(k)$ is {\em permissible} if for some (equivalently for all, by order indiscernibility) $q\in (0,1)\cap \Q$,
$\tp(\pi(\abar_q)/(N\setminus \abar_q))=\tp(\abar_q/(N\setminus \abar_q))$.
\end{definition}

Equivalently, $\pi$ is permissible if and only if the map sending $\abar_q$ to $\pi(\abar_q)$, and otherwise restricting to the identity, is an automorphism of $N$.

The following Lemma is easy because $Sym(k)$ is finite.

\begin{lemma}  \label{G}  There is a finite set $G\subseteq N^*$ such that for any $\pi\in Sym(k)$, $\pi$ is permissible if and only if
for some (equivalently, for every) $q\in (0,1)\cap \Q$,  $\tp(\pi(\abar_q)/G)=\tp(\abar_q/G)$.
\end{lemma}

\begin{proof}  Fix any $q\in (0,1)\cap\Q$.  For each $\sigma \in Sym(k)$ that is not permissible, choose a finite subset $G_\sigma^0\subseteq N\setminus\set{\abar_q}$
such that $\tp(\sigma(\abar_q)/G_{\sigma}^0)\neq \tp(\abar_q/G_{\sigma}^0)$.  By order indiscernibility, we may replace $G_\sigma^0$ by a `conjugate' $G_\sigma\subseteq N^*$
so that $\tp(\sigma(\abar_q)/G_{\sigma})\neq \tp(\abar_q/G_{\sigma})$.   Then, by order indiscernibility, $G:=\bigcup\set{G_\sigma:\sigma\in Sym(k), \text{ $\sigma$ not permissible}}$ works not only for $q$ but for any
$q'\in (0,1)\cap\Q$.
\end{proof}

Next, we pinpoint a failure of total indiscernibility over $M$.  
Since $\set{\abar_j:j\in\Q}$ is strictly order indiscernible over $M$ there is an integer $\ell\ge 2$, a 
permutation $\sigma\in Sym(\ell)$ and a formula $\theta(\xbar_1,\dots,\xbar_\ell,\mbar)$ (with $\mbar$ from $M$ and $\lg(\xbar_i)=k$ for each $i$)
such that 
\[N\models\theta(\abar_1,\dots, \abar_\ell,\mbar)\wedge\neg\theta(\abar_{\sigma(1)},\dots,\abar_{\sigma(\ell)},\mbar)\]
As $\sigma$ is a product of transpositions, this implies that there is some $i$, $1\le i<\ell$ such that
\[N\models\theta(\abar_1,\dots,\abar_{i-1},\abar_i,\abar_{i+1},\dots,\abar_\ell,\mbar)\wedge\neg
\theta(\abar_1,\dots,\abar_{i-1},\abar_{i+1},\abar_{i},\dots,\abar_\ell,\mbar)\]
Translating by $i$ and adding dummy variables as needed, there is some $r\ge 2$ such that
\[N\models\theta(\abar_{-r},\dots,\abar_{-1},\abar_0,\abar_1,\dots,\abar_r,\mbar)\wedge
\neg\theta(\abar_{-r},\dots,\abar_{-1},\abar_1,\abar_0,\dots,\abar_r,\mbar)\]
Let $H$ be the parameters $\set{\abar_{-r},\dots, \abar_{-1},\abar_2,\dots,\abar_r,\mbar}\subseteq N^*$
and let $\theta(\xbar,\ybar)$ be the $H$-definable formula mentioned above.

Take $F:=G\cup H\cup\set{\abar_0,\abar_1}$ to be our finite subset of $N^*$.
 Put $\gamma(\xbar):=\bigwedge \tp(\abar_q/F)$ for any $q\in (0,1)\cap \Q$.
 Let \[\delta(\xbar):=\theta(\xbar,\abar_1)\wedge\neg\theta(\xbar,\abar_0)\wedge(\xbar \cap F = \emptyset)\wedge\gamma(\xbar)\]

The following lemma characterizes when
$N\models\delta(\dbar)$ among all permutations of $\abar_q$.

\begin{lemma}  \label{indis}
(1) For $q,r\in [0,1]\cap\Q$, $N\models\theta(\abar_q,\abar_r)$ if and only if $q<r$.  

(2)  For $q\in\Q$ and $\pi\in Sym(k)$, $N\models \delta(\pi(\abar_q))$ if and only if $q\in (0,1)$ and $\pi$ is permissible.
\end{lemma}

\begin{proof}
(1) From above, this is true with $q=0, r=1$,
so the general statement follows by order indiscernibility.

(2)  Suppose $N\models\delta(\pi(\abar_q))$.  We first argue that $q\in (0,1)$.  
Note that $q=0,1$ are forbidden by $\gamma(\xbar)$.  
If $q<0$, then as $\<q,-r\dots,-1,0,2\dots,r\>$ has the same order type
as $\<q,-r,\dots,-1,1,2,\dots, r\>$,  indiscernibility yields
\[N\models\theta(\pi(\abar_q),\abar_0)\leftrightarrow\theta(\pi(\abar_q),\abar_1)\]
so $N\models\neg\delta(\pi(\abar_q))$.
Arguing similarly, $N\models\neg\delta(\pi(\abar_q))$ when $q>1$ as well.
Thus, $q\in (0,1)$.  But now, as $N\models\gamma(\pi(\abar_q))$ we have
$\tp(\pi(\abar_q)/G)=\tp(\abar_q/G)$, so $\pi$ is permissible by Lemma \ref{G}.

Conversely, suppose $q\in (0,1)$ and $\pi$ is permissible.
That $N\models\delta(\abar_q)$ follows from (1).  
As $\pi$ is permissible, $N\models\delta(\pi(\abar_q))$ as well.
\end{proof}

We next show
that  $N\models\neg\delta(\dbar)$ for any $\dbar\in N^k$ that is not a  permutation (permissible or otherwise) of some $\abar_q$. For this, we introduce the notion of a {\em hybrid}, which will be an $n$-tuple for some $n \leq k$ that is not (a permutation of) one of our ``intended'' tuples $\abar_q$. In future sections, we will make analogous definitions of ``unintended'' tuples and prove analogous lemmas to control their behavior.

\begin{definition}
Any automorphism $\sigma$ of $(\Q,\le)$ extends naturally to an automorphism
$\sigma^*\in Aut(N)$ that fixes $M$ pointwise, and maps each $\abar_q$ to $\abar_{\sigma(q)}$.
We call these automorphisms of $Aut(N)$ the {\em standard automorphisms}.  
\end{definition}

\begin{definition}  For any $n \leq k$, $\dbar\in N^n$ is a {\em hybrid} if no permutation of any $\abar_q$ is a subsequence of $\dbar$.
\begin{itemize}
\item  A hybrid $\dbar$ is {\em from $q_1<\dots<q_t$} if $\dbar\subseteq M\cup\abar_{q_1}\cup\dots\cup\abar_{q_t}$, and $\dbar\cap \abar_{q_i}\neq\emptyset$
for every $1\le i\le t$.
\item  If $\dbar$ is from $q_1<\dots<q_t$ and  $\dbar'$ is from $r_1<\dots<r_t$, we say $\dbar$ and $\dbar'$ are {\em associated} if 
$\sigma^*(\dbar)=\dbar'$ for some/any standard automorphism $\sigma^*\in Aut(N)$ extending any automorphism $\sigma\in Aut(\Q,\le)$ with $\sigma(q_i)=r_i$ for each $i$.
\end{itemize}

\end{definition}

The next lemma crucially uses that $M$ is array-minimal of index $k$.

\begin{lemma}  \label{stub}
Suppose $\bbar_q$ is a proper subsequence of $\abar_q$, $\bbar_r$ is a proper subsequence of $\abar_r$ and $\bbar_q$ and $\bbar_r$ are associated.
Then $\tp(\bbar_q/(N\setminus(\abar_q\cup\abar_r)))=\tp(\bbar_r/(N\setminus(\abar_q\cup\abar_r)))$.
\end{lemma}

\begin{proof}  Assume not.  Clearly, $q\neq r$, so assume $q<r$.  Choose a formula $\phi(\xbar,\ebar)$ with $\ebar\subseteq N\setminus(\abar_q\cup\abar_r)$ such that
\[N\models\phi(\bbar_q,\ebar)\wedge\neg\phi(\bbar_r,\ebar)\]
Choose a dense/codense subset $D\subseteq \Q$ and let $N_0$ be the substructure of $N$ with universe $M\cup\set{\abar_q:q\in (\Q\setminus D)}$.
Clearly, $N_0$ is an age-preserving extension of $M$, so we will obtain a contradiction to $M$ being array-minimal of index $k$ by proving that
$\tp(\bbar_{q'}/N_0)\neq\tp(\bbar_{r'}/N_0)$ for all pairs $q'<r'$ from $D$, where $\bbar_{q'}$ is the subsequence of $\abar_{q'}$ associated to both $\bbar_q$ and $\bbar_r$
and similarly for $\bbar_{r'}$.  (That each of these types is coordinate-wise non-algebraic is immediate, since each $\bbar_{q'}$ is disjoint from $N_0$.  Thus, each of these support
an infinite array by Lemma \ref{char}.)

To see this, fix $q'<r'$ from $D$, and let $\ebar$ be from $s_1 < \dots < s_t$.
As $D$ is dense/codense in $\Q$, there is some $\sigma \in Aut(\Q, \leq)$ sending $q \mapsto q'$, $r \mapsto r'$, and $s_1, \dots, s_t$ into $(\Q \bs D)$. Letting $\sigma^* \in Aut(N)$ be the corresponding standard automorphism, we have

\[N\models \phi(\bbar_{q'},\sigma^*(\ebar))\wedge\neg\phi(\bbar_{r'}, \sigma^*(\ebar))\]

As $\sigma^*(\ebar) \subset N_0$, we have $\tp(\bbar_{q'}/N_0)\neq\tp(\bbar_{r'}/N_0)$, as required.
\end{proof}

Next, we discuss arbitrary hybrids.  In the assumptions of the following lemma, the fact that $\dbar,\dbar'$ are associated implies that the $t$ is the same in both places.

\begin{lemma}  \label{trans}
For $n \leq k$, suppose $\dbar,\dbar'\in N^n$ are associated hybrids with $\dbar$ from $q_1<\dots<q_t$ and $\dbar'$ from $r_1<\dots<r_t$.
Then $\tp(\dbar/N_0)=\tp(\dbar'/N_0)$, where $N_0=N\setminus(\abar_{q_1}\cup\dots\abar_{q_t}\cup\abar_{r_1}\dots\cup\abar_{r_t})$.
\end{lemma}

\begin{proof} 
This will follow easily from the following special case.
\begin{claim*}
The statement holds if $\set{q_1\dots,q_t}$, $\set{r_1,\dots,r_t}$ are disjoint.  
\end{claim*}
\begin{claimproof}
Under this additional assumption, we argue by induction on $t$.
First, if $t=0$, then $\dbar\subseteq M$.  As $\dbar'$ is associated to $\dbar$, $\dbar'=\dbar$ so the statement is trivially true.

Now assume that the statement is true for $t-1$.  
Write $\dbar:=\hbar\bbar$, where $\hbar$ is from $q_1<\dots<q_{t-1}$ and $\bbar$ is from $q_t$.
Let $\sigma^*\in Aut(N)$ be a standard automorphism extending any automorphism $\sigma\in Aut(\Q,\le)$ extending the map $q_i\mapsto r_i$ for each $i$.
Let $\hbar':=\sigma^*(\hbar)$ and $\bbar':=\sigma^*(\bbar)$. As $\dbar$ is a hybrid, we have that $\bbar$ is a proper subsequence of $\abar_{q_t}$ (up to a permutation, which may be ignored), and so $\bbar'$ is also a proper subsequence of $\abar_{r_t}$, associated to $\bbar$.

To see that $\tp(\dbar/N_0)=\tp(\dbar'/N_0)$, choose any $\phi(\xbar,\ebar)\in \tp(\dbar/N_0)$.  Thus $N\models\phi(\hbar,\bbar,\ebar)$.  By our assumption that $\set{q_1,\dots,q_t}$ is disjoint from $\set{r_1,\dots,r_t}$, we have $\hbar\subseteq N\setminus(\abar_{q_t}\cup\abar_{r_t})$, and so
$N\models\phi(\hbar,\bbar',\ebar)$ by Lemma \ref{stub}.  But now, as $\hbar$ is a hybrid from $q_1<\dots<q_{t-1}$ that is associated to $\hbar'$, our inductive hypothesis
implies that $N\models\phi(\hbar',\bbar',\ebar)$.  Thus, $\phi(\xbar,\ebar)\in \tp(\dbar'/N_0)$ as needed.
\end{claimproof}

For the general case where $\set{q_1,\dots,q_t}$ and $\set{r_1\dots,r_t}$ need not be disjoint, choose any $\phi(\xbar,\ebar)\in\tp(\dbar/N_0)$.
Choose $s_1<\dots<s_t$ disjoint from $\set{q_1,\dots,q_t} \cup \set{r_1\dots,r_t}$ and such that $\ebar$ is disjoint from $\abar_{s_1}\cup\dots\cup\abar_{s_t}$.
Let $\dbar''$ be the hybrid from $s_1<\dots<s_t$ associated to both $\dbar$ and $\dbar'$.    Because of the disjointness, we can apply the
claim to the pairs $\dbar,\dbar''$ and $\dbar',\dbar''$ to obtain
\[N\models\phi(\dbar,\ebar)\leftrightarrow \phi(\dbar'',\ebar)\leftrightarrow \phi(\dbar',\ebar)\]
Thus, $\phi(\xbar,\ebar)\in\tp(\dbar'/N_0)$ as required.
\end{proof}

Finally, we can finish off our problem of identifying realizations of $\delta(\xbar)$ in $N^k$.

\begin{corollary}  \label{when}
For $\dbar\in N^k$,  $N\models\delta(\dbar)$ if and only if $\dbar=\pi(\abar_q)$ for some $q\in (0,1)\cap\Q$ and some permissible
$\pi\in Sym(k)$.
\end{corollary}

\begin{proof}  First, if $\dbar$ is $\pi(\abar_q)$ for some $q\in \Q$ and $\pi\in Sym(k)$, this is proved in Lemma \ref{indis}.
So assume $\dbar\in N^k$ is not a permutation of any $\abar_q$, i.e. $\dbar$ is a hybrid.  
We argue that $N\models\neg\delta(\dbar)$.  
Say $\dbar$ is from $q_1<\dots<q_t$.  
Choose $r_1<\dots<r_t<0$ from $\Q$, and let $\dbar'$ be associated to $\dbar$
from $r_1<\dots<r_t$.   By order indiscernibility,
\[N\models\theta(\dbar',\abar_0)\leftrightarrow \theta(\dbar',\abar_1)\]
In particular, $N\models\neg\delta(\dbar')$.
From the definition of $\delta(\xbar)$, we may assume $\dbar\cap F=\emptyset$, and so by Lemma \ref{trans} we
also have \[N\models \delta(\dbar)\leftrightarrow\delta(\dbar')\] so $N\models \neg\delta(\dbar)$ as claimed.
\end{proof}

The following lemma will finish the proof of Proposition \ref{one}.

\begin{lemma}  If $f:N_\alpha\rightarrow N_\beta$ is an isomorphism fixing $F$ pointwise, then
$(J_\alpha,\le)\cong(J_\beta,\le)$, hence $\alpha=\beta$.
\end{lemma}  

\begin{proof}  We define a map $f^*:J_\alpha\rightarrow J_\beta$ as follows.  Given $q\in J_\alpha$, note that
$N\models\delta(\abar_q)$.  Thus, $N\models \delta(f(\abar_q))$ as well.  By Corollary \ref{when} $f(\abar_q)=\pi(\abar_s)$
for some $s\in (0,1)$ and some permissible permutation $\pi$.  As $f(\abar_q)\subseteq N_\beta$, we must have
$s\in J_\beta$.  Put $f^*(q):=s$.  It is clear that $f^*:J_\alpha\rightarrow J_\beta$ is bijective.

To see that $f^*$ is order-preserving, choose $q<q'$ from $J_\alpha$.  Write $f(\abar_q)$ as $\pi(\abar_s)$ and write
$f(\abar_{q'})$ as $\pi'(\abar_{s'})$.  As both $\pi,\pi'$ are permissible, there is a $\sigma \in Aut(N)$ sending $\pi(\abar_s)\mapsto \abar_s$, $\pi'(\abar_{s'})\mapsto \abar_{s'}$, and fixing everything else.
Then the composition $g:= \sigma\circ f:N_\alpha\rightarrow N_\beta$ is an isomorphism fixing $F$ pointwise
sending $\abar_q\mapsto\abar_s$, $\abar_{q'}\mapsto \abar_{s'}$.

By Lemma \ref{indis}(1),
$N\models\theta(\abar_q,\abar_{q'})$.  As $\theta$ is quantifier-free, 
$N_\alpha\models\theta(\abar_q,\abar_{q'})$.  Since $g$ is an isomorphism fixing $F$ pointwise,
$N_\beta \models\theta(\abar_{s},\abar_{s'})$, and hence $N\models\theta(\abar_{s},\abar_{s'})$.    By Lemma \ref{indis}(1) again, $s<s'$.
That is, $f^*(q)<f^*(q')$.
\end{proof}

\section{$k$-cliques} \label{sec:kcliq}

In this section, we introduce $k$-cliques, which will serve the function of equivalence classes from Case 2 of \S \ref{sub:sketch}.

\medskip

\noindent{\bf Fix a finite relational $\LL$ with maximal arity $r$  and an ambient $\LL$-structure $M$ throughout this section.}

\medskip

For $n\ge r$, call a quantifier-free $\LL$-formula $\phi(x_1,\dots,x_n)$ {\em q.f.-complete} if $\phi(x_1,\dots,x_n)$ decides every atomic $R(\ybar)$ for every
permutation $\ybar$ of a subsequence of $(x_1,\dots,x_n)$.
As $\LL$ is finite relational, there is a finite set $\SS_n$ of q.f.-complete $\phi(x_1,\dots,x_n)$ such that for every $\LL$-structure $M$ and every $\cbar\in M^n$,
$\tp(\cbar)$ contains precisely one element of $\SS_n$.  Fix such a set $\SS_n$ for every $n\ge r$.

\begin{definition}   Fix $k\ge 1$ and let  $M^{(k)}:=\set{\abar\in M^k: a_i\neq a_j \text{ for }  i \neq j}$.   
\begin{itemize}
\item  A pair   $\abar,\bbar\in M^{(k)}$ is {\em exchangeable}, written $\abar \sim \bbar$, if  
$\abar\cap\bbar=\emptyset$ and $\tp(\abar\bbar/(M\setminus(\abar\cup\bbar)))=\tp(\bbar\abar/(M\setminus(\abar\cup\bbar)))$.
\item  A {\em $k$-clique} is a non-empty set $\A=\set{\abar_i:i\in I}\subseteq M^{(k)}$ such that $\abar_i,\abar_j$ are exchangeable whenever $i \neq j$.
\item The {\em size} of $\A$ is simply its
cardinality $|\A|$.
\item Given a $k$-clique $\AA$, we denote the set of all $a\in M$ such that $a\in\abar_i$ for some $\abar_i\in \AA$ by $\bigcup\A$.  
Because of the disjointness, $|\bigcup\A|=k\cdot|\A|$.
\end{itemize}
\end{definition}

\begin{remark}
Similar to Definition \ref{def:permis},  for all $\abar,\bbar\in M^{(k)}$ with $\abar\cap\bbar=\emptyset$,
$\abar$ and $\bbar$ are exchangeable if and only if the bijection swapping them is an automorphism of $M$ if and only if
\[M\models \forall\ybar[\ybar\cap(\abar\cup\bbar)=\emptyset\rightarrow \phi(\abar,\bbar,\ybar)\leftrightarrow\phi(\bbar,\abar,\ybar)]\]
for every $\phi(\xbar_1,\xbar_2,\ybar)\in \SS_{2k+r}$ with $\lg(\ybar)=r$.
As $\SS_{2k+r}$ is finite, it follows that exchangeability is definable on $M^{(k)}$.  
However,  unless $k=1$ exchangeability need not be transitive, due to the disjointness condition.
\end{remark}

\begin{definition}  A set of disjoint $k$-tuples $\A = \set{\abar_i : i \in I}\subseteq M^{(k)}$ is {\em totally indiscernible over its complement} if it is totally indiscernible over $M \bs \bigcup A$.
\end{definition}

\begin{lemma} \label{lemma:ti down}
Let $\AA\subseteq M^{(k)}$ be totally indiscernible over its complement, and let $\BB \subset \AA$. Then $\BB$ is totally indiscernible over its complement.
\end{lemma}
\begin{proof}
Let $\set{\bbar_1, \dots, \bbar_n}$, $\set{\bbar'_1, \dots, \bbar'_n}\subset \BB$ and let $\set{c_1, \dots, c_m} \subset M\bs\bigcup \BB$. By relabeling, let $\ell$ be such that $c_i \in \bigcup\AA$ iff $i \leq \ell$, and let $\abar_1, \dots, \abar_j \in \AA$ be such that $c_i \in \abar_1 \cup \dots \cup \abar_j$ for $i \leq \ell$. 

As $\AA$ is totally indiscernible over its complement, we have \[\tp(\bbar_1, \dots, \bbar_n,\abar_1, \dots, \abar_j/c_{\ell+1}, \dots c_m) = \tp(\bbar'_1, \dots, \bbar'_n,\abar_1, \dots, \abar_j/c_{\ell+1}, \dots c_m)\]
Thus, as desired, we have
\[\tp(\bbar_1, \dots, \bbar_n/c_{1}, \dots c_m) = \tp(\bbar'_1, \dots, \bbar'_n/c_{1}, \dots c_m)\]
\end{proof}

\begin{proposition} \label{prop:ti equiv}
 Let $\AA\subseteq M^{(k)}$ be pairwise disjoint.   Then $\AA$ is totally indiscernible over its complement if and only if $\AA$ is a $k$-clique.
\end{proposition}

\begin{proof}  
$(\Rightarrow)$ Suppose $\AA$ is totally indiscernible over its complement, and let $\abar_i, \abar_j \in \AA$. Then by Lemma \ref{lemma:ti down}, $\set{\abar_i, \abar_j}$ is totally indiscernible over its complement. Thus $\abar_i$ and $\abar_j$ are exchangeable.

$(\Leftarrow)$ Suppose $\AA = \set{\abar_i : i \in I}$ is a $k$-clique. Let $(i_1, \dots, i_n)$, $(i_1',\dots, i_n')\in I^n$. We proceed by induction on $m = |\set{\abar_{i_1}, \dots, \abar_{i_n}} \bs \set{\abar_{i'_1}, \dots, \abar_{i'_n}}|$.

If $m=0$ then there is some $\sigma \in Sym(n)$ such that $\sigma(i_1, \dots, i_n) = (i'_1, \dots, i'_n)$. As $\sigma$ can be written as a product of transpositions, we have $\tp(\abar_{i_1},\dots,\abar_{i_n}/(M\setminus\bigcup\A))=\tp(\abar_{i'_1},\dots,\abar_{i'_n}/(M\setminus\bigcup\A))$.

Now suppose $m = \ell+1$. After permuting the tuples, which we have seen does not affect their type, we may suppose $\abar_{i_1} \not\in \set{\abar_{i'_1}, \dots, \abar_{i'_n}}$ and $\abar_{i'_1} \not\in \set{\abar_{i_1}, \dots, \abar_{i_n}}$. Using that $a_{i_1}, a_{i'_1}$ are exchangeable for the first equality and the inductive hypothesis for the second, we have $\tp(\abar_{i_1},\dots,\abar_{i_n}/(M\setminus\bigcup\A))=\tp(\abar_{i'_1},\abar_{i_2}, \dots,\abar_{i_n}/(M\setminus\bigcup\A)) = \tp(\abar_{i'_1},\dots,\abar_{i'_n}/(M\setminus\bigcup\A))$.
\end{proof}

\begin{lemma} \label{meld}
Suppose $\AA$ and $\BB$ are $k$-cliques, $\AA\cap\BB \neq \emptyset$, and $\bigcup (\AA \bs \BB) \cap \bigcup(\BB\bs \AA) = \emptyset$. Then $\AA\cup\BB$ is a $k$-clique.
\end{lemma}
\begin{proof}
First, we show distinct tuples $\abar, \bbar \in \AA \cup \BB$ are disjoint. If $\abar, \bbar \in \AA$ (or $\abar, \bbar \in \BB$, this follows from the definition of $k$-cliques. Otherwise $\abar \in (\AA \bs \BB)$ and $\bbar \in (\BB \bs \AA)$, and so are disjoint by the last assumption.

 Let $\abar \in (\AA\setminus\BB)$, $\bbar \in (\BB\setminus\AA)$, and choose $\cbar \in \AA \cap \BB$. Let $Y = M \bs (\abar\cup\bbar\cup\cbar)$. By a sequence of transpositions, each involving $\cbar$, we have
\[\tp(\abar\bbar\cbar/Y) = \tp(\abar\cbar\bbar/Y) = \tp(\cbar\abar\bbar/Y) = \tp(\bbar\abar\cbar/Y)\]
Thus $\tp(\abar\bbar/Y\cbar) = \tp(\bbar\abar/Y\cbar)$, and so $\abar \sim \bbar$, as desired.
\end{proof}

Infinite $k$-cliques $\A$ in $M$ give rise to types that support infinite arrays.  
\begin{definition}Let $\AA$ be an infinite $k$-clique and let $\xbar=(x_1,\dots,x_k)$. The {\em average type of $\AA$}, written $Av_{\AA}(\xbar)$, is the set
\[\set{\phi(\xbar,\ebar): \phi \text{ is q.f., } \ebar\in M^{<\omega}, \hbox{$M\satisfies \phi(\abar,\ebar)$ for some/all $\abar\in\AA$ with $\abar\cap\ebar=\emptyset$}}\]
\end{definition}
\begin{lemma}  \label{average}
If $\AA$ is an infinite $k$-clique in $M$, then $Av_\AA(\xbar)$ is well-defined and  $Av_{\A}(\xbar)\in\Arr_k(M)$.
\end{lemma}

\begin{proof} 
For well-definedness, we must check the ``some/all'' claim implicit in the definition. As $\AA$ is an infinite $k$-clique, $\abar,\abar'\in\AA$ are exchangeable, hence $\tp(\abar/\ebar)=\tp(\abar'/\ebar)$ whenever $\abar \cap \ebar = \emptyset$. It is easily verified that it is a complete (quantifier-free) type over $M$.  As any finite subset of $Av_{\AA}(\xbar)$
is realized by infinitely many $\abar\in\AA$,  we see that
$Av_{\AA}(\xbar)\in \Arr_k(M)$.
\end{proof}

\medskip

\centerline{{\bf For the remainder of this section, fix an integer $k\ge 1$.}}

\medskip

\begin{definition}  Let $M$ be any $\LL$-structure.
\begin{itemize}
\item  For any  $k'\le k$, call a $k'$-clique $\AA$ in $M$ {\em sufficiently large} if $|\AA|>2k+r$.
\item
 An extension $N\supseteq M$ is {\em $(\le k)$-clique-preserving} if, for every $k'\le k$, 
 every sufficiently large $k'$-clique $\A$ in $M$  remains a $k'$-clique in $N$.
\end{itemize}
\end{definition}

We will see two ways of obtaining $(\le k)$-clique-preserving extensions of $M$.  The first follows from the definability of exchangeability.

\begin{remark} \label{rem:pres}
If $M^*\succeq M$, then since exchangeability is definable, $M^*$ will be both age-preserving and $(\le k)$-clique-preserving.  
Moreover, any substructure $N$ satisfying $M\subseteq N\subseteq M^*$
will also be an age-preserving, $(\le k)$-clique preserving extension of $M$.    
\end{remark}

The second method involves extending existing, sufficiently large cliques.

\begin{definition}  Fix an $\LL$-structure $M$ and recall $k$ is fixed throughout.
\begin{enumerate}
\item  A {\em simple clique extension of $M$} is an extension $N$ with universe $M\cup\bigcup\CC$, where for some $k'\le k$, $\CC$ is a $k'$-clique in $N$ extending some
sufficiently large $k'$-clique $\AA$ in $M$.
\item  A {\em clique extension of $M$} is  a continuous, nested union $\bigcup N_\alpha$ of simple clique extensions, i.e., $N_0=M$,
$N_{\alpha+1}$ is a simple clique extension of $M$, and $N_\lambda=\bigcup_{\alpha<\lambda} N_\alpha$ for limit $\lambda$.
\end{enumerate}
\end{definition}

\begin{lemma}  \label{safeextend}  Every clique extension $N\supseteq M$ is $(\le k)$-clique preserving.
\end{lemma}
\begin{proof} Arguing by induction on the length of the chain, it suffices to show this when $N$ is a simple clique extension of $M$.
Similarly, 
arguing by induction on $|\CC\setminus\AA|$, it suffices to show this when $\CC=\AA\cup\set{\cbar}$ and $N=M\cup\set{\cbar}$.
So choose any $k'\le k$ and any $k'$-clique $\BB$ in $M$.  To see that $\BB$ remains a $k'$-clique in $N$,
choose $\bbar,\bbar'\in\BB$ and $\hbar\in (N\setminus(\bbar\cup\bbar'))^r$.  It suffices to show that $N\models\phi(\bbar,\bbar',\hbar)\leftrightarrow\phi(\bbar',\bbar,\hbar)$
for every $\phi\in \SS_{2k'+r}$.   Write $\hbar=\cbar'\ebar$, where $\cbar'=\hbar\cap\cbar$ and $\ebar=\hbar\setminus\cbar$ (so $\ebar\subseteq M$).
As $\AA$ is sufficiently large,  choose $\abar\in\AA$ disjoint from $\bbar\bbar'\hbar$ and let $\abar'\subseteq\abar$ be the subsequence corresponding to $\cbar'$ in $\cbar$.
As $\ebar\abar'$ are from $M$, $\bbar\sim \bbar'$ in $M$, and as $\phi$ is quantifier-free, we have
\[N\models\phi(\bbar,\bbar',\ebar,\abar')\leftrightarrow \phi(\bbar',\bbar,\ebar,\abar')\]
Since $\cbar\sim\abar$ in $N_0$ and $\bbar\bbar'\ebar$ is disjoint from $\cbar\abar$,  we conclude $N\models\phi(\bbar,\bbar',\ebar,\cbar')\leftrightarrow \phi(\bbar',\bbar,\ebar,\cbar')$,
as required.
\end{proof}

Consider the case of an equivalence relation with infinitely many infinite classes from \S \ref{sub:sketch}. This was easier than the general non-mutually algebraic case. For an example closer to the general case, consider when $M$ is an equivalence relation with infinitely many infinite classes, as well as infinitely many classes of each finite size. If we proceed as in \S \ref{sub:sketch}, each $M_f$ will be isomorphic to $M$. In this case, the problem is easily remedied by first passing to an age-preserving
$M' \supset M$ in which every class is infinite. In the general case, this may not be possible, but we may find some age-preserving $M' \supset M$ in which every (sufficiently large) maximal finite $k$-clique cannot be extended further. This is the notion of fullness discussed next. Carrying out the construction from \S \ref{sub:sketch} over this $M'$, we will be able to differentiate the maximal finite $k$-cliques that come from shrinking some infinite $\AA_q$ from $M^*$ with those that were already in $M'$, since only the former will be infinitely extendable.

It is easily seen by Zorn's Lemma that inside every $M$, every $k'$-clique $\AA$ in $M$ is contained in a maximal $k'$-clique $\BB\supseteq\AA$ in $M$.
What is less clear is whether a maximal $k'$-clique $\AA$ can be extended in some age-preserving extension $N\supseteq M$. 

\begin{definition}  Fix an $\LL$-structure $M$.
\begin{enumerate}
\item   For $k'\le k$, call a $k'$-clique $\AA$ in $M$ {\em infinitely extendable} if there is some
age-preserving $N\supseteq M$ and an infinite $k'$-clique $\CC\supseteq \AA$ in $N$; and call 
$\AA$ {\em unextendable} if it is maximal in every age-preserving $N\supseteq M$.
\item  
$M$  is {\em $k$-full} if, for every $k'\le k$, 
 every sufficiently large,  maximal $k'$-clique $\AA$ in $M$,  $\AA$ is either infinite or unextendable.
\end{enumerate}
\end{definition}

Clearly, if a $k'$-clique $\AA$ is not infinitely extendable, then there is an age-preserving $N\supseteq M$ and an unextendable (finite)  $k'$-clique $\CC$ in $N$ extending $\AA$.
In fact, we can additionally require that the age-preserving extension be $(\le k)$-clique preserving as well.

\begin{lemma} \label{stepone}  Suppose $M$ is a countable
$\LL$-structure, and for some $k'\le k$,  $\AA$ is a sufficiently large $k'$-clique in $M$.  Then there is an age-preserving, $(\le k)$-clique-preserving
countable $N\supseteq M$ and an extension $\CC\supseteq\AA$ such that:
\begin{enumerate}  
\item  If $\AA$ is infinitely extendable, then $\CC$ is infinite; and
\item  If $\AA$ is not infinitely extendable, then  $\CC$ is  unextendable.
\end{enumerate}
\end{lemma}  

\begin{proof}   In both cases, choose an age-preserving $N^*\supseteq M$ and a $k'$-clique $\CC$ in $N^*$ extending $\AA$ that is either infinite, or of largest possible finite size.
In either case, let $N$ be the substructure of $N^*$ with universe $M\cup\bigcup\CC$.  Then $N$ is also an age-preserving extension of $M$, and moreover $N$ is a clique extension.
Thus, $N$ is a $(\le k)$-clique preserving extension of $M$ by Lemma \ref{safeextend}.
\end{proof}

The following lemma now follows by bookkeeping.

\begin{lemma} \label{full}  Every countable structure $M$ has a countable,  $k$-full, $(\le k)$-clique-preserving, age-preserving  extension $N\supseteq M$.
\end{lemma}

\begin{proof}  We first claim that given any countable $M$, there is a countable, age-preserving, $(\le k)$-clique-preserving $M'\supseteq M$ such that for each $1\le k'\le k$,
each of the (countably many) sufficiently large, finite $k'$-cliques $\AA$ in $M$ has an extension $\CC\supseteq M'$ that is either infinite or is unextendable. ($M'$ is obtained as union of a countable chain of age-preserving, $(\le k)$-clique-preserving extensions formed by iterating Lemma \ref{stepone} once for each such $\AA$.)

Now, simply iterate the claim above $\omega$ times, getting a nested sequence $M=M_0\subseteq M_1\subseteq M_2\subseteq \dots$ with $M_{n+1}=(M_n)'$ from above.
Then $N=\bigcup M_n$ is as desired.
\end{proof}

\section{Grid extensions} \label{sec:full}

We now generalize the construction of adding infinitely many new equivalence classes from Case 2 of \S \ref{sub:sketch}.
Throughout this section, we will work with in a finite, relational  language $\LL$ with arity bounded by $r$ and we will be considering
non-mutually algebraic models that are array-minimal of index $k$ (recall Definition \ref{arraymin}). {\bf These $k$ are $r$ are fixed throughout this section.}
Thus, e.g., a $k'$-clique $\AA$ will be sufficiently large if $|\AA|>2k+r$. 

\begin{lemma} \label{step}
 Suppose $M$ is not mutually algebraic, $M$ is array-minimal of index $k$, there is no age-preserving $N\supseteq M$ with $2^{\aleph_0}$ siblings,  and let $p\in \Arr_k(M)$. 
Then there is an age-preserving, clique preserving $N\supseteq M$ containing an infinite $k$-clique $\A=\set{\abar_\ell:\ell\in\omega}$ with each $\abar_\ell$ realizing $p$.
\end{lemma}

\begin{proof}  As $p\in \Arr_k(M)$, we can use Ramsey's theorem and compactness to find an elementary extension $M^*\succeq M$ containing an order-indiscernible over $M$
sequence $\<\abar_\ell:\ell\in \omega\>$ of realizations of $p$.  This sequence must be totally indiscernible over $M$, as otherwise Proposition \ref{one} would give an age-preserving $N\supseteq M$
with $2^{\aleph_0}$ siblings.  Take $N$ to be the substructure of $M^*$ with universe $M\cup\set{\abar_\ell:\ell\in\omega}$.
As $\A=\set{\abar_\ell:\ell\in\omega}$ is totally indiscernible over its complement, it is a $k$-clique by Proposition \ref{prop:ti equiv}.  The fact that $N$ is age-preserving and clique preserving follows by Remark \ref{rem:pres}.
\end{proof}

\begin{lemma}  \label{omega}  Suppose $M$ is not mutually algebraic, $M$ is array-minimal of index $k$, and there is no age-preserving $N\supseteq M$ with $2^{\aleph_0}$ siblings.
Then there is an $R(\xbar,\ybar) \in \LL$, an infinite set $\set{p_q:q\in \Q}\subseteq \Arr_k(M)$, 
a tuple $\dbar_{q,r}\in M^{\lg(\ybar)}$ for all $q<r \in \Q$, and an age-preserving, clique-preserving $N\supseteq M$
with infinite $k$-cliques $\set{\A_q :q\in \Q}$ from $N$
 such that, letting $\A_q = \set{\abar_{q, i}: i \in \omega}$, the following hold.
 
 \begin{enumerate}
 \item $\bigcup \A_q \cap \bigcup \A_r = \emptyset$ for $q \neq r$.
 \item   For each $q \in \Q$ and $i \in \omega$, $\abar_{q,i}$ is a realization of $p_q$.
 \item   For each $q<r  \in \Q$ and $i\in\omega$, $N\models R(\abar_{q,i},\dbar_{q,r})\wedge\neg R(\abar_{r,i},\dbar_{q,r})$.
 \end{enumerate}
 \end{lemma}
 
 \begin{proof}  First fix a sequence $\<p_i:i\in\Q\>$ of distinct complete $k$-types over $M$, each of which support an infinite array.
As the types are distinct, for each $i<j<\omega$ there is an $R_{i,j}(\xbar, \ybar_{i,j}) \in \LL$ and $\dbar_{i,j}$ from $M$ such that $R(\xbar,\dbar_{i,j})$ is in $p_i$ but not in $p_j$.
As $\LL$ is finite, by Ramsey's theorem we can choose a specific $R(\xbar,\ybar)$ and an infinite $I\subseteq \Q$ such that $R_{i,j}=R$ whenever
$i<j$ from $I$.  Because of this, Clause (3) follows immediately from Clause (2).  

We construct $N$ in $\omega$ steps, once for each $i\in I$,
each time
applying Lemma \ref{step} to the type $p_i$.  Because each of the extensions are clique-preserving, the union of this sequence suffices.
\end{proof}

\begin{definition} \label{def:grid} \leavevmode
\begin{itemize}
\item 
Fix $R(\xbar,\ybar) \in \LL$. A {\em $(k, R)$-grid extension over $M$} is an age-preserving $N \contains M$ satisfying the following conditions.
\begin{enumerate}
\item $N = M \cup \set{\abar_{q, i} \in N^k: q \in \Q, i \in \omega} \cup \set {\dbar_{q, r}: q < r \in \Q}$.
\item The $\abar_{q, i}$ are pairwise disjoint and disjoint from $M$.
\item For each $q \in \Q$, $\AA_q = \set{\abar_{q,i}:i \in \omega}$ is a $k$-clique.
\item For all $q<r \in \Q$ and $i\in\omega$,
$N\models R(\abar_{q,i},\dbar_{q,r})\wedge\neg R(\abar_{r,i}, \dbar_{q,r})$.
\end{enumerate}
\item Let $\ebar_{q,r} = \dbar_{q,r} \bs (M \cup \bigcup_{q\in\Q}(\bigcup \AA_q))$. Given any order-automorphism $\sigma\in Aut(\Q,\le)$, let $\sigma^*$ be the bijection on $N$ defined as follows.
\begin{enumerate}
\item  For $q\in\Q$, $\sigma^*(\abar_{q,i})=\abar_{\sigma(q),i}$;
\item  For $q<r$ from $\Q$, $\sigma^*(\ebar_{q,r})=\ebar_{\sigma(q),\sigma(r)}$
\item  $\sigma^*$ fixes $M$ pointwise.
\end{enumerate}

\item  An {\em indiscernible $(k,R)$-grid extension} is a $(k,R)$-grid extension $N\supseteq M$ such that, for every $\sigma\in Aut(\Q,\le)$,
the induced
$\sigma^*$ is an automorphism of $N$.  We call such $\sigma^*$ a {\em standard automorphism of $N$}, and any composition of $\sigma^*$ with an element of $\Pi_{q \in \Q} Sym(\AA_q)$ a {\em permuted standard automorphism of $N$}.
\end{itemize}
\end{definition}

\begin{proposition} \label{quote}  Suppose $M$ is not mutually algebraic, $M$ is array-minimal of index $k$, and there is no age-preserving extension $N\supseteq M$ with $2^{\aleph_0}$ siblings.
Then there is an indiscernible $(k, R)$-grid extension $N \contains M$.
\end{proposition}
\begin{proof}
We proceed by compactness. Expand the language by constant symbols naming every element of $M$, as well as $k$-tuples of constant symbols $\set{\abar_{q,i} : q \in \Q, i \in \omega}$ and $\ell$-tuples of constant symbols $\set{\dbar_{q, r} : q<r \in \Q}$, where $\ell$ is the length of $\dbar_{q,r}$ in Lemma \ref{omega}. Consider the theory $T^*$ in this language:
\begin{enumerate}
\item The elementary diagram of $M$.
\item The $\abar_{q,i}$ are pairwise disjoint, and no element from $M$ is in any such tuple.
\item For $q<r \in \Q$, $R(\abar_{q, 0}, \abar_{r,0}, \dbar_{q,r}) \wedge \neg R(\abar_{r, 0}, \abar_{q,0}, \dbar_{q,r})$.
\item Each $\AA_q = \set{\abar_{q, i} : i \in \omega}$ is a $k$-clique, and is order indiscernible over all the other constants.
\item For every $\sigma \in Aut(\Q, \leq)$, let $\sigma^*$ be the induced self-bijection of $M \cup \set{\abar_{q, i} : q \in \Q, i \in \omega} \cup \set{\dbar_{q, r}: q<r \in \Q}$. Then $\sigma^*$ is an automorphism.
\end{enumerate}

Models of finite subsets of $T^*$ are obtained by applying the finite Ramsey theorem to the model from Lemma \ref{omega}. Thus, by compactness, we obtain a model $M^* \models T^*$. Taking the restriction of $M^*$ to the constant symbols, and letting $N$ be the reduct to the original language, we are finished.
\end{proof}

\begin{definition}
Let $N \supset M$ be an indiscernible $(k,R)$-grid extension. For $q<r \in \Q$, let $\ebar_{q,r}$ be as in Definition \ref{def:grid}. By indiscernibility, each $\ebar_{i,j}$ must be the same length.

Define the {\em rank} of $N \supseteq M$ to be the length of any $\ebar_{i,j}$. It is possible for the rank to be 0.
\end{definition}

\begin{example}
Let $M$ consist of an equivalence relation with infinitely many infinite classes, and let $N = M \cup \set{a_{q,i}: q \in \Q, i \in \omega}$, where each $\AA_q = \set{a_{q,i}:i \in \omega}$ is a new class. Then we may take $\dbar_{q,r} = a_{q,0}$, giving rank 0.

Our next example codes equivalence relations in a different language. Take $M$ in a language $(U,V,R)$, where $U,V$ are unary and $R$ is binary. Let $U$ and $V$ be infinite and partition $M$, and let $R$ be such that for each $u \in U$ there is a unique $v \in V$ such that $R(u,v)$, and for each $v \in V$ there are infinitely $u \in U$ such that $R(u,v)$. Let $N = M \cup \set{u_{q,i}: q \in \Q, i \in \omega} \cup \set{v_q:q \in \Q}$, where each $u_{q,i} \in U$, $v_q \in V$, and $R(u_{q,i}, v_r)$ holds if $q=r$. Taking $\AA_q = \set{u_{q,i}:i \in \omega}$ and $\dbar_{q,r} = v_q$ gives rank 1. We could not have given this extension rank 0, as $\set{u_{q,i}: q \in \Q, i \in \omega}$ is totally indiscernible over $M$; the $v_q$'s are needed to break them into distinct $k$-cliques.
\end{example}

We now show that in an indiscernible $(k, R)$-grid extension of minimum rank, each $\AA_i$ is a maximal $k$-clique.

\begin{definition}
Let $N \supset M$ be an indiscernible $(k, R)$-grid extension. Two tuples $\abar_1 \subset \abar_{q, i}, \abar_2 \subset \abar_{r, j}$ are {\em associated} if the natural bijection between $\abar_{q,i}$ and $\abar_{r, j}$ maps $\abar_1$ to $\abar_2$.
\end{definition}

The next lemma is analogous to Lemma \ref{stub}. 

\begin{lemma} \label{lemma:stub2}
Suppose $M$ is not mutually algebraic, $M$ is array-minimal of index $k$, and $N \supset M$ is an indiscernible $(k, R)$-grid extension. Suppose $\abar_1 \subsetneq \abar_{q, i}, \abar_2 \subsetneq \abar_{r, j}$ are associated. Then $\tp(\abar_1/(N \bs (\abar_{q,i} \cup \abar_{r,j}))) = \tp(\abar_2/(N \bs (\abar_{q,i} \cup \abar_{r,j})))$.
\end{lemma}
\begin{proof}
We may assume $q \neq r$, since otherwise this follows from $\abar_{q, i} \sim \abar_{q, j}$, and for definiteness take $q < r$. By indiscernibility, it suffices to prove this assuming $i = j = 0$. Let $N_0 = N \bs \set{\abar_{\ell, 0} : \ell \in \Q}$.

\begin{claim*}
$\tp(\abar_1 / N_0) = \tp (\abar_2 / N_0)$.
\end{claim*}
\begin{claimproof}
Each standard automorphism fixes $N_0$ setwise. Suppose $\tp(\abar_1 / N_0) \neq \tp (\abar_2 / N_0)$, as witnessed by $\wbar$. Then for any $\sigma \in Aut(\Q, \leq)$, the standard automorphism $\sigma^*(\wbar)$ witnesses that $\tp(\sigma^*(\abar_1) / N_0) \neq \tp (\sigma^*(\abar_2) / N_0)$. But this contradicts that $M$ is array-minimal of index $k$.
\end{claimproof}

Now suppose $\wbar$ witnesses that $\tp(\abar_1/(N \bs (\abar_{q,0} \cup \abar_{r,0}))) \neq \tp(\abar_2/(N \bs (\abar_{q,0} \cup \abar_{r,0})))$. Let $\pi \in \Pi_i Sym(\AA_i)$ be such that $\pi(\wbar) \in N_0$, and $\pi$ fixes $\abar_{q, 0}$ and $\abar_{r, 0}$. Then $\pi(\wbar)$ witnesses that $\tp(\abar_1 / N_0) \neq \tp (\abar_2 / N_0)$, contradicting the Claim.
\end{proof}

\begin{lemma} \label{lemma:ai max}
Suppose $M$ is not mutually algebraic,  $M$ is array-minimal of index $k$, and $N \supset M$ is an indiscernible $(k, R)$-grid extension of minimum rank. For a given $q \in \Q$ and $\hbar \in N^k$, $\hbar \sim \abar_{q, 0}$ only if $\hbar$ is a permutation of $\abar_{q,i}$ for some $i$.

In particular, for every $q$, $\AA_q = \set{\abar_{q,i} : i \in \omega}$ is a maximal $k$-clique.
\end{lemma}
\begin{proof}
Fix $q \in \Q$, and suppose $\hbar \in N^k$ is not a permutation of some $\abar_{q, i}$. Let $N = M\sqcup A \sqcup E$, where $A = \bigcup_i(\bigcup \AA_i)$ and $E = N \bs (M \cup A)$. The proof splits into two cases.

\textbf{Case 1:} $\hbar \cap E \neq \emptyset$. Let $\ebar_{s, t} \subset E$ be such that $\ebar^h_{s,t} = \hbar \cap \ebar_{s,t} \neq \emptyset$, and let $\ebar'_{s,t} = \ebar_{s,t} \bs \ebar^h_{s,t}$. As $\hbar \sim \abar_{q,0}$, let $\abar_{q,0}^h \subset \abar_{q,0}$ correspond to the entries of $\ebar^h_{q,0}$. Let $\dbar_{s,t}$ witness that $\cbar_{s,0} \not\sim \cbar_{t,0}$, with $\ebar_{s,t} \subset \dbar_{s,t}$. Let $\dbar^*_{s,t}$ be obtained by replacing $\ebar_{s,t}$ with  $\cbar^h_{q,0} \ebar'_{s,t}$. Let $\ell$ be large enough that none of the tuples mentioned so far intersect $\abar_{s, \ell}$ or $\abar_{t, \ell}$. We will show $\dbar^*_{s,t}$ still witnesses that $\cbar_{s,\ell} \not\sim \cbar_{t, \ell}$, contradicting the fact that $N$ has minimum rank.

By taking an automorphism replacing $\abar_{q,0}$ with some $\abar_{q,i}$, we may assume $\dbar_{s,t} \cap \abar_{q,0} = \emptyset$. Let $\dbar'_{s,t} = \dbar_{s,t} \bs \ebar_{s,t}$. Since $\hbar \sim \abar_{q,0}$, $\tp(\hbar/\abar_{s, \ell}\abar_{t, \ell} \ebar'_{s,t}\dbar'_{s,t}) = \tp(\abar_{q,0}/\abar_{s, \ell}\abar_{t, \ell} \ebar'_{s,t}\dbar'_{s,t})$. Thus $\tp(\ebar^h_{s,t}/\abar_{s, \ell}\abar_{t, \ell} \ebar'_{s,t}\dbar'_{s,t}) = \tp(\abar^h_{q,0}/\abar_{s, \ell}\abar_{t, \ell} \ebar'_{s,t}\dbar'_{s,t})$, and so $\tp(\dbar_{s,t}/\abar_{s, \ell}\abar_{t, \ell}) = \tp(\dbar^*_{s,t}/\abar_{s, \ell}\abar_{t, \ell})$.

\textbf{Case 2:} $\hbar \cap E = \emptyset$. Given an interval $[x, y)$ in $\omega$, we define $A\upharpoonright_{[x, y)} = \bigcup \set{\abar_{q, i} : q \in \Q, i \in [x, y)}$. Choose $\ell_1$ such that $\hbar \cap A \subset A \upharpoonright_{[0, \ell_1)}$. Fix $r > q$, and let $\wbar$ witness $\abar_{q, 0} \not\sim \abar_{r, 0}$. By permuting each $\AA_i$, we may choose $\ell_2 > \ell_1$ so that $\wbar \subset A \upharpoonright_{[\ell_1, \ell_2)}$. For any $\ell \geq \ell_2$, we have $\wbar$ also witnesses $\abar_{q, \ell} \not\sim \abar_{r, \ell}$. Let $N_0 = N \bs (A \upharpoonright_{[0, \ell_1)})$. We use $\xbar \sim_{N_0} \ybar$ to mean $\xbar$ and $\ybar$ are exchangeable over $N_0$, i.e. for any $\zbar$ from $N_0$, $\tp(\xbar\ybar\zbar) = \tp(\ybar\xbar\zbar)$. 

\begin{claim*}
$\hbar \sim_{N_0} \abar_{r, \ell}$.
\end{claim*}
\begin{claimproof}
As $h \cap E = \emptyset$, let $\hbar \subset \nbar \abar_{t_1, i_1} \dots \abar_{t_j, i_j} = \gbar$, where $\nbar = \hbar \cap M$, each $i < \ell_1$, and $t_1 \leq \dots \leq t_j$. Let $s_1 \leq \dots \leq s_j < q$, let $\gbar_2 = \nbar \abar_{s_1, i_1} \dots \abar_{s_j, i_j}$, and let $\hbar_2 \subset \gbar_2$ be associated with $\hbar$. By Lemma \ref{lemma:stub2}, we have $\tp(\hbar/N_0) = \tp(\hbar_2/N_0)$. In particular, $\tp(\hbar/\cbar_{q, \ell}\cbar_{r, \ell}\dbar) = \tp(\hbar_2/\cbar_{q, \ell}\cbar_{r, \ell}\dbar)$, for all $\dbar \subset N_0$.

Thus we have $\hbar \sim_{N_0} \abar_{q, \ell} \iff \hbar_2 \sim_{N_0} \abar_{q, \ell}$, and similarly for $\abar_{r, \ell}$. By assumption, $\hbar \sim \abar_{q, \ell}$, so we also have $\hbar \sim_{N_0} \abar_{q, \ell}$. Now let $\sigma \in Aut(\Q, \leq)$ be an automorphism with $\sigma(q) = r$ and fixing all $s \leq s_{j}$, and let $\sigma^*$ be the corresponding standard automorphism. This shows $\hbar_2 \sim_{N_0} \abar_{r, \ell}$, and so we also have $\hbar \sim_{N_0} \abar_{r, \ell}$. 
\end{claimproof}

We now handle the fact that $\hbar$ might intersect $\wbar$. As we took $\wbar \in A \upharpoonright_{[\ell_1, \ell_2)}$, and $\hbar \cap E = \emptyset$, we have $\mbar = \hbar \cap \wbar \subset M$. Let $\hbar = \hbar'\mbar$ and $\wbar = \wbar'\mbar$. Then
\[\tp(\abar_{q, \ell}\abar_{r, \ell}\wbar'\hbar) = \tp(\hbar\abar_{r, \ell}\wbar'\abar_{q, \ell}) =  \tp(\abar_{r, \ell}\hbar\wbar'\abar_{q, \ell}) = \tp(\abar_{r, \ell}\abar_{q, \ell}\wbar'\hbar)\]
where we have used $\hbar \sim \abar_{q, \ell}$ in the first and third equalities, and $\hbar \sim_{N_0} \abar_{r, \ell}$ in the second.

Removing $\hbar'$ from the initial and final expressions, and noting $\wbar = \wbar'(\hbar \bs \hbar')$, we contradict that $\wbar$ witnesses $\abar_{q, \ell} \not\sim \abar_{r, \ell}$.
\end{proof}

\begin{definition}
Let $N \supset M$ be an indiscernible $(k, R)$-grid extension. A $k$-clique $\BB=\set{\bbar_s:s\in I} \subset N^k$ is {\em homogeneous} if each $\bbar_s\in\BB$ can be partitioned into $\nbar_s \mbar_s$ (with either part of the partition possibly empty) satisfying the following.
\begin{enumerate}
\item   $\nbar_s$ is from $(N\setminus M)$, $\mbar_s$ is from $M$.
\item  For each $1\le t\le k$, for all $s,s'\in I$, $(\bbar_s)_t\in M$ iff $(\bbar_{s'})_t\in M$.
\item  For all $s,s'\in I$ there is some permuted standard automorphism $\sigma^*$ such that $\sigma^*(\nbar_s)=\nbar_{s'}$.
\end{enumerate}
\end{definition}

\begin{lemma} \label{good} Suppose that $M$ is $k$-full and that $N \supset M$ is an indiscernible $(k, R)$-grid extension. There is a constant $C'$ so that if $\B$ is a maximal $k$-clique in $N$ that has size at least $C'$ and is infinitely extendable,
then $\B$ is already infinite.
\end{lemma}

\begin{proof} By two applications of the pigeonhole principle, we can compute a $C'$ so that any $k$-clique of size $C'$ contains a homogeneous $k$-clique $\B_0$ with $|\B_0| \geq 2$. The result will follow by infinitely iterating the following claim to show $\B_0$ is infinitely extendable. By Lemma  \ref{safeextend} $\BB$ will remain a $k$-clique in the corresponding clique-extension, and so be infinitely extendable by Lemma \ref{meld}.
\begin{claim*}
  Suppose $\B_0 \subset N$ is a finite, homogeneous, infinitely extendable
 $k$-clique of size at least 2.  Then there is a proper extension $\B_1\supsetneq\B_0$ that is also 
homogeneous.
\end{claim*}

\begin{claimproof}  First, since $\B_0=\set{\nbar_s\mbar_s:s\in I}$ is a $k$-clique in $N$, the subsequences $\set{\mbar_s:s\in I}$ form an $\ell$-clique in $M'$, where $\ell=\lg(\mbar)$.
Because $\B_0$ is infinitely extendable, so is $\set{\mbar_s:s\in I}$.  As $M'$ is $\ell$-full, we can find some $\mbar^*$ so that $\set{\mbar_s:s\in I}\cup\set{\mbar^*}$
is an $\ell$-clique in $M'$, and thus in $N$, as $M' \subset N$ is a $k$-clique-preserving extension. (If $\mbar_s$ is empty, this may be ignored.)

Choose a permuted standard automorphism $\pi \in Aut(N)$ such that $\pi$ fixes $\nbar_0$ and $\pi(\nbar_1)$ is disjoint from $\bigcup \BB_0$ (the existence of $\pi$ uses the homogeneity of $\BB_0$). Let $\nbar^*:=\pi(\nbar_1)$.
We claim that $\B_0\cup\set{\nbar^*\mbar^*}$ is a homogeneous $k$-clique.  The homogeneity is clear from the construction. We now show $\set{n_0m_0, n^*m^*}$ is a $k$-clique, and that $\B_0\cup\set{\nbar^*\mbar^*}$ is a $k$-clique will follow by Lemma \ref{meld}.
\begin{align*}
\tp(\nbar^*\mbar^*\nbar_0\mbar_0/(N \bs \nbar^*\mbar^*\nbar_0\mbar_0)) &=
\tp(\nbar^*\mbar_1\nbar_0\mbar_0/(N \bs \nbar^*\mbar_1\nbar_0\mbar_0)) \\
&= \tp(\nbar_1\mbar_1\nbar_0\mbar_0/(N \bs \nbar_1\mbar_1\nbar_0\mbar_0)) \\
&= \tp(\nbar_0\mbar_0\nbar_1\mbar_1/(N \bs \nbar_1\mbar_1\nbar_0\mbar_0)) \\
&= \tp(\nbar_0\mbar_0\nbar^*\mbar_1/(N \bs \nbar^*\mbar_1\nbar_0\mbar_0)) \\
&= \tp(\nbar_0\mbar_0\nbar^*\mbar^*/(N \bs \nbar^*\mbar^*\nbar_0\mbar_0))
\end{align*}
We have used that $\set{\mbar_1, \mbar^*}$ is an $\ell$-clique in lines 1 and 5, applied $\pi^{-1}$ to get to line 2, used that $\set{n_0m_0, n_1m_1}$ is a $k$-clique to get to line 3, and applied $\pi$ to get to line 4. 
\end{claimproof}
\end{proof}

\section{Non-mutually algebraic $T$} \label{sec:notma}

\begin{theorem}  \label{thm:not ma}
If $M$ is a non-mutually algebraic model of $T$, then there is an age-preserving $N\supseteq M$ with $2^{\aleph_0}$ siblings.
\end{theorem}

\begin{proof}
First take an age-preserving $M'' \supseteq M$ that is array-minimal of index $k$, by Lemma \ref{ez}. Then by Lemma \ref{full}, let $M' \contains M''$ be a $k$-full age-preserving, $k$-clique-preserving extension. Suppose $M'$ has no age-preserving extension with $\cont$ siblings, and by Proposition \ref{quote}, let $N \contains M'$ be an indiscernible $(k, R)$-grid extension over $M'$, for some $R \in \LL$, of minimum rank. We will show $N$ has $\cont$ siblings, which is a contradiction.

Choose a dense/codense subset $D \subseteq \Q$, and let $D^c = \Q \bs D$. Using the notation of Definition \ref{def:grid}, let
$N_{D^c}$ be the substructure of $N$ with universe $M'\cup\set{\abar_{i,\ell}:i\in D^c,\ell\in\Q}\cup \set{\ebar_{i,j}:i<j, i,j\in D^c}$.
By the indiscernibility, $N_{D^c}$ is isomorphic to $N$ over $M'$.  Thus, any  model $N^*$ satisfying $N_{D^c}\subseteq N^* \subseteq N$ is a sibling of $N$,  in fact via embeddings that fix $M'$ pointwise.

Let $r$ be the maximum arity of the language, let $C'$ be from Lemma \ref{good}, and choose $C$ such that any $k$-clique of size at least $C$ contains a homogeneous $k$-clique of size $\max(C', 2k+r)$. Given an injective $f \colon D \to \omega \bs [C]$, we construct $N_f \subset N$ by restricting $\AA_q$ to a subset $\AA^*_q$ of size $f(q)$, for each $q \in D$. It remains to show the $N_f$ are pairwise non-isomorphic. The following claim is sufficient, as being an infinitely extendable $k$-clique of size $n$ is type-definable.

\begin{claim*} \label{claim:Nf}
For any $n\ge C$, $N_f$ has an infinitely extendable maximal $k$-clique of size $n$ if and only if $n\in Im(f)$.
\end{claim*}
\begin{claimproof}
$(\Leftarrow)$ Let $q \in \Q$ be such that $f(q) = n$. First, $N$ is visibly a clique extension of $N_f$, hence $N$ is $(\le k)$-clique-preserving by Lemma \ref{safeextend}. Thus, as $\AA_q$ is a maximal $k$-clique in $N$ by Lemma \ref{lemma:ai max}, $\AA^*_q$ is a maximal $k$-clique in $N_f$. As it is infinitely extendable to $\AA_q$, we are finished.

$(\Rightarrow)$ This will follow immediately from Lemma \ref{lemma:only ai}.
\end{claimproof}
\end{proof}

\begin{lemma} \label{lemma:only ai}
Let $C \in \omega$, $D \subset \Q$, $N_f$, and $\set{\AA^*_q: q \in D}$ be as in the proof of Theorem \ref{thm:not ma}. If $\BB \subset (N_f)^k$ is a finite infinitely extendable maximal $k$-clique of size at least $C$, then there is some $q \in D$ such that each element of $\BB$ is a permutation of some element of $\AA^*_q$.
\end{lemma}
\begin{proof}
Suppose not. We now work within $N_f$. Suppose $|\BB| \geq C$, let $n=\max(C', 2k+r)$ (where $C'$ is from Lemma \ref{good} and $r$ is the maximum arity of the language), and let $\set{\bbar_i:i <n} = \BB^- \subset \BB$ be a homogeneous $k$-clique. We first prove the conclusion for $\BB^-$. There must be some $q \in D$ such that $\bigcup \BB^-$ intersects $\bigcup \AA^*_q$; otherwise $\BB^-$ would be infinitely extendable by Lemma \ref{good}. Pick such a $q$. There is at least one $j$ such that $\bbar_0 \cap \abar_{q,j} \neq \emptyset$, so let $\cbar_0 = \bbar_0 \cap \abar_{q,j}$, and let $\lg(\cbar_0) = k' < k$ (this inequality is strict by our assumption that $\bbar_0$ is not a permutation of $\abar_{q,j}$). For each $i$, let $\cbar_i$ be the subtuple of $\bbar_i$ associated with $\cbar_0$, and let $\CC = \set{\cbar_i:i < n}$. By relabeling, we may assume $\cbar_i = \bbar_i \cap \abar_{q,i}$.

\begin{claim*}
$\CC$ is a $k'$-clique.
\end{claim*}
\begin{claimproof}
Suppose $\cbar_0 \not\sim \cbar_1$, as witnessed by $\wbar$, with $\lg(\wbar) \leq r$. Then $\wbar \cap (\bbar_0 \cup \bbar_1) \neq \emptyset$, otherwise $\wbar$ would witness $\bbar_0 \not\sim \bbar_1$.

As $\BB^-$ is sufficiently large, by relabeling we may suppose $\wbar$ does not intersect $\bbar_2 \cup \bbar_3$. Let $\pi$ be the automorphism swapping $\abar_{q,0}$ with $\abar_{q,2}$ and swapping $\abar_{q,1}$ with $\abar_{q,3}$, while fixing everything else. Then $\pi(\wbar)$ witnesses $\cbar_2 \not\sim \cbar_3$, but $\pi(\wbar) \cap (\bbar_2 \cup \bbar_3) = \emptyset$, which is a contradiction.
\end{claimproof}

Now work in $N$, and note that $\CC$ remains a $k'$-clique in $N$ by Lemma \ref{safeextend}, since $N$ is a clique extension of $N_f$. 
For each $r \in \Q$, let $\sigma^*_r$ be a standard automorphism sending $\AA_q$ to $\AA_r$.  
Each $\sigma^*_r(\CC)$ is a $k'$-clique  that extends to an infinite $k'$-clique within $N$.
However,  for $r_1 \neq r_2$, $\sigma^*_{r_1}(\cbar_0) \not\sim \sigma^*_{r_2}(\cbar_0)$, since $\sigma^*_{r_1}(\abar_0) \not\sim \sigma^*_{r_2}(\abar_0)$, so the average types of these infinite extensions are distinct.  Thus, by Lemma \ref{average}, we conclude that $\Arr_{k'}(N)$ is infinite, 
contradicting that $M$ is array-minimal of index $k$.

Given the conclusion for $\BB^-$, it follows for $\BB$ by Lemma \ref{lemma:ai max}.
\end{proof}

\section{Mutually algebraic $T$} \label{sec:ma}

\subsection{The non-cellular case}
In this subsection, we prove that if $M$ is mutually algebraic but non-cellular, then it admits a countable elementary extension with $\cont$ siblings.

If $\LL$ is finite relational and $M$ is mutually algebraic, then by Theorem \ref{thm:ma at}, there is another finite relational language $\LL'$ in which every atomic relation is mutually algebraic, and such that $\LL'$ is quantifier-free interdefinable with an expansion of $\LL$ naming finitely many constants.  

Adding finitely many constants to our language changes our sibling count by at most a factor of $\az$, and so will not affect this subsection. Adding the constants and switching language to $\LL'$ as above, we may assume the following.

\medskip

\noindent \textbf{For this subsection, we assume $M$ is mutually algebraic in a finite relational language with mutually algebraic atomic relations.}

\begin{definition}

Given $M$ in a language with mutually algebraic atomic relations, we may construct a corresponding hypergraph $G_M$ on the same universe, placing an edge on a tuple $\mbar$ if $R$ holds on (some permutation of) $\mbar$ for some $R \in \LL$.

We call $A\subseteq M$ an {\em MA-connected part} if $A$ is a connected part of $G_M$.

Equivalently, we may use that if $\delta(x,\ybar)$ and $\theta(x,\zbar)$ are quantifier-free mutually algebraic with at least one variable symbol $x$ in common, then $\delta(x,\ybar)\wedge\theta(x,\zbar)$ is quantifier-free, mutually algebraic. Then $A\subseteq M$ is an {\em MA-connected part} iff, for all $a,b\in A$, there are $\set{c_2,\dots,c_n}\subseteq A$ and a quantifier-free mutually algebraic $\phi(x,y,\zbar)$ such that $M\models\phi(a,b,c_2,\dots,c_n)$

An {\em MA-connected component} is a maximal MA-connected part.
\end{definition}

\begin{lemma} \label{break}
The following points follow from the corresponding facts for connected parts of hypergraphs. 
\begin{enumerate}
\item  If $A,B\subseteq M$ are MA-connected parts and $A\cap B=\emptyset$, then $A\cup B$ is an MA-connected part.
\item Every MA-connected part is contained in a unique MA-connected component.
\item If $C$ is an infinite MA-connected part, there is a nested sequence $B_0 \subsetneq B_1 \subsetneq \dots$ such that $\cup_i B_i = C$ and each $B_i$ is a finite MA-connected part.
\end{enumerate}
\end{lemma}

Suppose $M$ and $N$ are siblings.  Then $Age(M) = Age(N)$ and so if $M$ thinks that $\delta(x_1,\dots,x_n)$ is mutually algebraic, then $N$ also thinks this.
Using this fact, we have:

\begin{lemma}  Suppose $M$ and $N$ are siblings and $f:M\rightarrow N$ is an embedding.
Then for any MA-connected part $A\subseteq M$, $f(A)$ is an MA-connected part of $N$.
Thus, if $C\subseteq M$ is an MA-connected component, then $f(C)$ is contained in an MA-connected component
as well.
\end{lemma}

\begin{lemma}  \label{suff}  Suppose $M$ is mutually algebraic and there is an infinite set $\set{C_i:i\in\omega}$ of components
such that for each $i$, $C_i$ properly embeds into $C_{i+1}$, but there is no embedding of $C_{i+1}$ into $C_i$.
Then $M$ has $2^{\aleph_0}$ siblings.
\end{lemma}

\begin{proof}  Call an MA-connected component $Z$ {\em outside the scope} if there is no embedding of $Z$ into any $C_i$.  
Let $Z^*=\bigcup\set{Z: Z \text{ is outside the scope}}$. Note that any MA-connected component inside the scope embeds into all but finitely many $C_i$.  For each  infinite $S\subseteq\omega$, let
$N_S$ be the substructure of $N$ with universe $Z^*\cup\set{C_i:i\in S}$.  

We first argue that each $N_S$ is a sibling of $M$.  Fix any infinite $S\subseteq\omega$.  
Enumerate the MA-connected components $\set{Y_j:j\le\omega}$ of $M$ that are within the scope.
Inductively define a mapping $h:M\rightarrow N_S$ as the union of a chain of mappings $\<h_n:n\in\omega\>$
as follows.  Let $h_0:Z^*\rightarrow N_S$ be the identity.  Assume that $h_j:N^*\cup\set{Y_t:t<j}\rightarrow N_S$ has been defined.
Given $Y_j$, choose some $i$ not already chosen so that $Y_j$ embeds into $C_i$, and let $h_{j+1}$ extend $h_j$ by mapping $Y_j$ into $C_i$.

To see the $N_S$ are pairwise non-isomorphic, note that $N_S$ contains an MA-connected component isomorphic to $C_i$ iff $i \in S$. As isomorphisms must map MA-connected components to MA-connected components, we are finished.
\end{proof}

\begin{lemma} \label{cor} If $M$ contains infinite, pairwise isomorphic MA-connected components $\set{C_i : i \in \omega}$, then $M$ has $\cont$ siblings.
\end{lemma}

\begin{proof} We will produce a sibling $N$ of $M$ satisfying the hypotheses of Lemma \ref{suff}, which
suffices.

Let $X \subset \omega$ be infinite/co-infinite. We will produce $N$ by shrinking each $C_i$ with $i \in X$. We will have that $M$ embeds into $N$ as we leave an infinite collection of $C_i$ unaltered.

As $C_0$ is infinite, by Lemma \ref{break}
write $C_0=\bigcup \set{B_i:i\in \omega}$, where each $B_i$ is a finite, MA-connected part and $B_i \subsetneq B_{i+1}$ for each $i$. We now construct $N \subset M$ by restricting $C_i$ down to an isomorphic copy of $B_i$, for each $i \in X$.
\end{proof}

\begin{theorem}[\cite{cma}] \label{thm:cma}
Let $\LL$ be finite relational, and suppose $M$ is a  mutually algebraic but non-cellular countable $\LL$-structure. Then there is some $M^* \succ M$ such that $M^*$ contains infinitely many new infinite MA-connected components, pairwise isomorphic over $M$.

Furthermore, we may take the universe of $M^*$ to be the universe of $M$ together with these new components.
\end{theorem}

\begin{proposition}  If $M$ is not cellular then there is an age-preserving extension $N$ with $2^{\aleph_0}$ siblings.  In the case where
$M$ is mutually algebraic, $N$ can be chosen to be an elementary extension of $M$.
\end{proposition}

\begin{proof}  Suppose $M$ is not cellular.  If $M$ is not mutually algebraic, then we are done by Theorem \ref{thm:not ma}.
  
If $M$ is mutually algebraic but non-cellular, then produce $M^* \succ M$ as in Theorem \ref{thm:cma}. By Lemma \ref{cor}, $M^*$ has $\cont$ siblings.
\end{proof}

\subsection{The cellular case}

In this subsection, we will be able to directly consider the siblings of $M$, rather than of some age-preserving extension.

\begin{example}
Consider the cellular graph $M$ consisting the disjoint union of infinitely many disconnected edges and an infinite independent set. Here, we may obtain $\az$ siblings as follows. First, we pass to the subgraph $N$ removing the independent set, which will be a sibling of $M$. Then, for each $i \in \omega$, we obtain a sibling $N_i$ by removing a point from $i$ of the edges.
\end{example}

Note that in a cellular partition (Definition \ref{def:cell}), for a fixed $i \in [n]$, $\set{\cbar_{i,j}: j \in \omega}$ is a $k_i$-clique.

\begin{definition}
 A cellular partition is {\em separated} if for every $i \in [n]$, there is no proper subtuple of $\cbar_{i,0}$ such that the set of associated subtuples amongst $\set{\cbar_{i,j}: j \in \omega}$ forms a $k$-clique.
\end{definition}

Given a cellular partition, we may always produce a separated cellular partition by increasing $n$ and splitting apart any offending tuples.

Suppose $M$ is cellular, with  cellular partition $K \cup \set{\cbar_{i,j}: i \in [n], j \in \omega}$. Given some $\cbar_{i,j}$ and $S \subseteq [k_i]$, let $\cbar_{i,j}^S = (c_{i,j}^\ell | \ell \in S) \subseteq \cbar_{i,j}$. Then every substructure $N \subseteq M$ is specified by $N \cap K$ as well as, for each $i \in [n]$ and $S \subseteq [k_i]$, the number of $j$ such that $N \cap \cbar_{i,j} = \cbar_{i,j}^S$.

Recall that $M$ is finitely partitioned if and only if $|\cbar_{i,j}|=1$ for every $i$.

\begin{lemma}
If $M$ is cellular and not finitely partitioned, then $M$ has $\az$ siblings.
\end{lemma}
\begin{proof}
By the discussion above, a cellular structure has at most $\az$ siblings. Let $K \cup \set{\cbar_{i,j} : i \in [n], j \in \omega}$ be a separated cellular partition of $M$. As $M$ is not finitely partitioned, there is some $i$ such that $|\cbar_{i,j}| > 1$. Fix some $\ell \in \omega$, for each $i,j$ let $c_j$ be the first element of $\cbar_{i,j}$, and let $M_\ell = M \bs \set{\cbar_{i,j}\bs c_j: j \leq \ell}$. For any $i'$ such that $|\cbar_{i',j}| = 1$ and $\set{c_j : j \leq \ell} \cup \set{\cbar_{i',j}: j \in \omega}$ is a $1$-clique, remove all $\cbar_{i',j}$, and let $M^*_\ell$ be the resulting structure. Note $M^*_\ell$ is a sibling of $M$.

We now show there is no $m \in M^*_\ell \bs K$ such that $m \sim c_j$ for some $j \leq \ell$. Suppose there is, and $m$ is the $k^{th}$ element of $\cbar_{i',j'}$ for some $i' \in [n]$ and $j' \in \omega$.  Then $c_j$ will be exchangeable with the $k^{th}$ element of $\cbar_{i',j''}$ for every $j'' \in \omega$, and so these elements will form a $1$-clique. If $|\cbar_{i', j'}| = 1$, this contradicts the construction of $M^*_\ell$. If $|\cbar_{i', j'}| >1$, this contradicts that we started with a separated cellular partition.

Let $C_\ell$ be the maximal $1$-clique in $M^*_\ell$ containing $\set{c_j: j \leq \ell}$. Then $C_\ell \subseteq K \cup \set{c_j: j \leq \ell}$ by the previous paragraph, so $\ell \leq |C_\ell| \leq |K| +\ell$. In $M^*_\ell$, any $1$-clique containing a point outside $K \cup \set{c_j: j \leq \ell}$ is either a singleton or infinite, since, as in the previous paragraph, if $x \sim y$ where $y$ is the $k^{th}$ coordinate of $\cbar_{i',j'}$, then $x$ is exchangeable with the $k^{th}$ element of $\cbar_{i',j''}$ for every $j'' \in \omega$. Thus for $\ell > |K|$, $C_\ell$ will be the largest maximal finite $1$-clique of $M^*_\ell$ . By the bounds above on $|C_\ell|$, if $\ell' > |K|+\ell$, then $|C_{\ell'}| > |C_{\ell}|$, and so $M^*_{\ell'} \not\cong M^*_{\ell}$, since their largest maximal finite $1$-cliques have different sizes. Thus, by varying $\ell$, we may produce $\az$ siblings of $M$.
\end{proof}

\begin{lemma}
If $M$ is finitely partitioned, then $M$ has one sibling, namely itself.
\end{lemma}
\begin{proof}
As $M$ is $\omega$-categorical, it admits an $\omega$-categorical model-companion $M^*$ \cite{Sar}. Then $M^*$ is a sibling of $M$, so it suffices to show $M^*$ has only one sibling. 

As being finitely partitioned is a universal property, $M^*$ is also finitely partitioned, and so admits a cellular partition with $K = \acl(\emptyset)$, and $|\cbar_{i,j}| = 1$ for each $i \in [n]$, so let $c_{i,j}$ be the one element of $\cbar_{i,j}$. We may further assume that we have taken $n$ minimal (subject to $|\cbar_{i,j}| = 1$), and thus $tp(c_{i,j}/K) \neq tp(c_{i',j}/K)$ for $i \neq i'$.

As $M^*$ is model-complete, every $x \in K$ is algebraic by an existential formula, so any substructure with the same age must contain all of $K$. The age of $M^*$ also specifies $\set{c_{i,j}}$ is infinite for each $i$, so any substructure with the same age is isomorphic to $M^*$.
\end{proof}

\subsection{The main theorem}

Putting together the results of this section, we have our main theorem.

\begin{theorem} \label{thm:main2}
Let $T$ be a universal theory in a finite relational language. Then one of the following holds.
\begin{enumerate}
\item $T$ is finitely partitioned. Every model of $T$ has one sibling.
\item $T$ is cellular. The finitely partitioned models of $T$ have one sibling and the non-finitely partitioned models have $\aleph_0$ siblings.
\item $T$ is not cellular. For every non-cellular $M \models T$, there is some $N \supseteq M$ such that $N \models T$ and $N$ has $2^{\aleph_0}$ siblings. Furthermore, if $T$ is mutually algebraic, we may take $N \succeq M$.
\end{enumerate}
\end{theorem}

If $T$ admits a structure universal for its age, this immediately gives the following corollary.

\begin{corollary} \label{cor:universal2}
Let $M$ be a countable model in a finite relational language that is universal for its age. Then one of the following holds.
\begin{enumerate}
\item $M$ is finitely partitioned, and has one sibling.
\item $M$ is cellular but not finitely partitioned, and has $\aleph_0$ siblings.
\item $M$ is not cellular, and has $\cont$ siblings.
\end{enumerate}
\end{corollary}

A weakening of ``finite relational language'' is given in the following definition.

\begin{definition}
We say {\em $M$ has finite profile} if, for every $n$, the number of isomorphism types of substructures of size $n$ is finite.
\end{definition}

We now show the assumption of a finite relational language in Corollary \ref{cor:universal2} cannot be weakened to finite profile.

\begin{example} \label{ex:inf lang}
Let the language consist of one $n$-ary relation symbol $R_n$ for each $n \in \omega$. Let $\xbar_n = (x_n^1, \dots, x_n^n)$. Let $M = \bigsqcup_{n \in \omega} \xbar_n \sqcup \bigsqcup_{n \in \omega} y_n$, where $R_n(\xbar)$ holds iff $\xbar = \xbar_n$, and the $y_n$ form an independent set.

$M$ is not $\omega$-categorical, as $x_n^i$ and $x_m^j$ have different (non-quantifier-free) 1-types for $n \neq m$. For each $n$, the isomorphism type of $n$ points is determined by which tuples $\xbar_i$ for $i \leq n$ they contain, and so $M$ has finite profile. That $M$ is universal for its age is clear by inspection.

Age-preserving extensions of $M$ can only add further points to the independent set, and so the only sibling of $M$ is itself. As $M$ is not $\omega$-categorical, it is not finitely partitioned, nor even cellular.
\end{example}

As noted in \cite{sib}, Corollary \ref{cor:universal2} implies the same conclusion with the hypothesis that $M$ is universal for its age replaced with the hypothesis that $M$ is $\omega$-categorical, since we may then pass to the model companion of $M$.

We also obtain a positive answer to a question from \cite{sib} as another corollary of our result. The proof simply goes through each case of Theorem \ref{thm:main2}, which immediately implies the corresponding case of the corollary.

\begin{corollary}
	For an age $\Age$ in a finite relational language, let $(Mod(\Age), \leq)$ be the countable structures with age $\Age$, quasi-ordered by embeddability. Then for every $M \in Mod(\Age)$, the number of structures $\leq$-above $M$ is equal to $|Mod(\Age)|$.
\end{corollary}

\section{Open questions}

\setcounter{conjecture}{0}

\begin{conjecture} [Thomass\'{e}, \cite{Thom}] \label{conj:Thom}
Given a countable structure $M$ in a countable relational language, $M$ has either 1, $\az$, or $\cont$ siblings, up to isomorphism.
\end{conjecture}

As mentioned in the introduction, Conjecture \ref{conj:Thom} seems outside the scope of the model-theoretic approach of this paper. However, an interesting special case to consider may be when $M$ is mutually algebraic. After naming finitely many constants, we may decompose $M$ into MA-connected components, which seem easy to analyze. However, the effect of naming the constants is mysterious.

\begin{problem}
Confirm Conjecture \ref{conj:Thom} when $M$ is mutually algebraic.
\end{problem}

As noted in the introduction, the arguments in this paper bear out the following intuition: if a universal theory $T$ is non-cellular, then either it is unstable and so has a model encoding $(\Q, <)$, or has a model that in some sense encodes an infinite partition, i.e. a partition with infinitely many infinite parts.

\begin{question} \label{q:enc}
	What is the proper notion of ``encodes an infinite partition'' to formalize the intuition above?
\end{question}

Even attempting to plausibly refine Conjecture \ref{conj:Thom} to describe which structures fall into which of the three cases seems difficult, but answering Question \ref{q:enc} may be helpful. We know that there are two reasons for a universal theory to have a model with $\cont$ siblings: either there is a model encoding a linear order with $\cont$ siblings (namely $(\Q, <)$) or a model ``encoding an infinite partition''.  Perhaps the same is essentially true at the level of individual models, although we must weaken the requirement of an infinite partition, since an equivalence relation with arbitrarily large finite classes also has $\cont$ siblings.

\begin{question} \label{q:cont}
	If a countable relational structure $M$ has $\cont$ siblings, must $M$ either encode a linear order with $\cont$ siblings, or either ``encode an infinite partition'' or ``encode a partition with arbitrarily large finite parts''  in the sense of Question \ref{q:enc}?
\end{question}

From \cite{chains}, we know exactly which countable linear orders have $\cont$ siblings; furthermore, the linear orders with $\cont$ siblings seem to either encode infinite partitions or partitions with arbitrarily large finite parts.

The final section of \cite{sib} and the introduction of \cite{MPW} contain several open problems, some of which we mention below.

A positive answer to the following conjecture would answer Problem 2 of \cite{MPW}. As mentioned there, Lachlan has proven that an age $\Age$ has a unique countable model up to elementary equivalence iff $\Age$ is finitely partitioned \cite{Lach}.

\begin{conjecture} \label{conj:elem eq}
All cases of Theorem \ref{thm:main2} can be strengthened to pairwise non-elementarily equivalent siblings. In particular, given an age $\mathfrak A$, there are $\cont$ non-elementarily equivalent countable structures of age $\Age$ iff $\mathfrak A$ is non-cellular.
\end{conjecture}

The place where our proof falls short of this conjecture is that whether a $k$-clique is infinitely extendable does not seem to be definable. However, in some cases, considering infinite extendability is unnecessary; for example, if $M$ has only finitely many 1-types, in particular if $M$ is $\az$-categorical, then there is a bound $C$ on the size of $k$-cliques appearing in $M$. When constructing $N_f$ in Theorem \ref{thm:not ma}, we may always shrink our $k$-cliques above $C$, and distinguish $N_f$ from $N_g$ by whether it has a maximal $k$-clique of some particular size above $C$. Thus we have proven Conjecture \ref{conj:elem eq} in the case $\Age$ is the age of an $\az$-categorical structure.

Given an age $\Age$, let $Mod(\Age)/{\equiv}$ denote the bi-embeddability classes of countable structures with age $\Age$. Thomass\'{e}'s conjecture is concerned with the size of any single $\equiv$-class.  There are several conjectures regarding the number of $\equiv$-classes in \cite{sib}, from which we mention the following.

\begin{conjecture} [\cite{sib}]
For an age $\Age$ in a finite relational language, $|Mod(\Age)/{\equiv}|$ is finite if and only if $|Mod(\Age)/{\equiv}| = 1$ if and only if $\Age$ is cellular.
\end{conjecture}

If the conjecture above is true, then the only possibilities for $|Mod(\Age)/{\equiv}|$ are $\set{1, \az, \aleph_1, \cont}$ \cite{sib}. Classifying which ages fall into which case would be a natural next step.

For problems involving model-counting in an age, such as in this paper or the problem of determining $|Mod(\Age)|$ in \cite{MPW}, the dividing lines are preserved under arbitrary expansions by (finitely many) unary relations. This is clear after proving that these dividing lines correspond to being finitely partitioned or being cellular. However, if this could be proven as a first step, then the approach taken in this paper could be drastically simplified, since a non-mutually algebraic theory admits a model such that in a unary expansion there is a definable equivalence relation on singletons with infinitely many infinite classes. We then would not have to use grid extensions to mimic the behavior of such an equivalence relation, and would not have to worry about hybrid tuples.

\begin{question}
	Let $M$ be a countable structure in a finite relational language, and let $M^*$ be an expansion by finitely many unary relations. Let $\Age$ and $\Age^*$ be their respective ages. Can any of the following statements be proven without first classifying the dividing lines?
	\begin{enumerate}
		\item If $|Mod(\Age^*)| = \cont$, then $|Mod(\Age)|= \cont$.
		\item If $Mod(\Age^*)$ has a structure with $\cont$ siblings, then so does $Mod(\Age)$.
		\item If $|Mod(\Age^*)/{\equiv}|$ is infinite, then so is $|Mod(\Age)/{\equiv}|$.
	\end{enumerate}
\end{question}

  \bibliographystyle{asl}
  \bibliography{SiblingsBib}

\end{document}